\documentclass[12pt]{article}
\usepackage{amsmath,amsxtra,amssymb,latexsym,epsfig,amscd,amsthm,subfigure,fancybox,epsfig}
\usepackage[mathscr]{eucal}
\usepackage{graphicx}
\usepackage{epsfig} 
\usepackage{epstopdf}
\usepackage{cases}
\usepackage{subfig}
\usepackage{titlesec}
\usepackage{multicol,color}
\titleformat{\subsection}[runin]
        {\normalfont\bfseries}
        {\thesubsection}
        {0.5em}
        {}
        [.]
\newtheorem{thm}{Theorem}[section]
\newtheorem {asp}{Assumption}[section]

\newtheorem{lem}{Lemma}[section]
\newtheorem{prop}{Proposition}[section]
\theoremstyle{definition}

\theoremstyle{remark}
\newtheorem{rem}{Remark}
\newtheorem{example}{Example}[section]
\numberwithin{equation}{section}

\newcommand{\eps}{\varepsilon}

\newcommand{\A}{\mathcal{A}}

\newcommand{\M}{\mathcal{M}}
\newcommand{\F}{\mathcal{F}}

\newcommand{\E}{\mathbb{E}}

\newcommand{\N}{\mathbb{N}}
\newcommand{\PP}{\mathbb{P}}

\newcommand{\K}{\mathcal{K}}

\newcommand{\R}{\mathbb{R}}

\newcommand{\Lom}{\mathcal{L}}

\newcommand{\abs}[1]{\left\vert#1\right\vert}

\numberwithin{equation}{section}

\newcommand{\bed}{\begin{displaymath}}
\newcommand{\eed}{\end{displaymath}}
\newcommand{\bea}{\bed\begin{array}{rl}}
\newcommand{\eea}{\end{array}\eed}

\newcommand{\barray}{\begin{array}{ll}}
\newcommand{\earray}{\end{array}}

\newcommand{\1}{\boldsymbol{1}}

\def\bar{\overline}
\def\hat{\widehat}
\def\a.s{\text{\;a.s.\;}}

\begin{document}
\title{ Permanence and Extinction for the Stochastic SIR Epidemic Model}
\author{N. H. Du\thanks{Department of Mathematics, Mechanics and
Informatics, Hanoi National University,
 334 Nguyen Trai, Thanh Xuan, Hanoi Vietnam, dunh@vnu.edu.vn. This research was
supported in part by Vietnam National Foundation for Science and Technology Development   (NAFOSTED)
n$_0$  101.03.2017.308} \, and \, N. N. Nhu \thanks{Corresponding author: Department of Mathematics, Wayne State University, Detroit MI 48202 USA, nguyen.nhu@wayne.edu, nhu.math.2611@gmail.com.}
}
\maketitle

\begin{abstract}
The  aim of this paper is to study the stochastic SIR equation with general incidence functional responses and in which both natural death rates and the incidence rate are perturbed by white noises. We derive a sufficient and almost necessary condition for the extinction and permanence for SIR epidemic system with multi noises
\begin{equation*}
\begin{cases}
dS(t)=\big[a_1-b_1S(t)-I(t)f(S(t),I(t))\big]dt + \sigma_1 S(t) dB_1(t) -I(t)g(S(t),I(t))dB_3(t),\\
dI(t)=\big[-b_2I(t) + I(t)f(S(t),I(t))\big]dt + \sigma_2I(t) dB_2(t) + I(t)g(S(t),I(t))dB_3(t).
\end{cases}
\end{equation*}
Moreover, the rate of all convergences of the solution are also established. A number of numerical examples are given to illustrate our results.



\medskip
\end{abstract}

\setlength{\baselineskip}{0.28in}
\section{Introduction}\label{sec:int}
The epidemic models have a very long history and have been widely studied because of their importance in ecology. Such models were first introduced by Kermack and McKendrick in \cite{KER1,KER2} and recently, much attention has been devoted to analyzing, predicting the spread and designing controls of infectious
diseases in host populations; see \cite{Lou2017,DANG,NHU1,NHU2,HUO,Dli-17,Dli-18,NNY20,NY19,RUAN,SHAN,SHAN1,YANG,ZHANG}. One of the classical epidemic models is the SIR model which is suitable for modeling some diseases with permanent immunity such as rubella, whooping cough, measles, smallpox, etc.
The SIR  epidemic models consist of three groups of  individuals: the susceptible, infected and recovered individuals, whose densities at the time $t$ are denoted by $S(t), I(t)$ and $R(t)$, respectively. The relationship between these quantities are in general described by the following equations
\begin{equation}\label{1.0}
\begin{cases}
dS(t)=\big[a_1-\mu_SS(t)-F(S(t),I(t))\big]dt, \\
dI(t)=\big[-(\mu_I+r)I(t) + F(S(t),I(t)\big]dt,\\
dR(t)=\big[-\mu_RR(t)+rI(t)\big]dt,
\end{cases}
\end{equation}
where $a_1>0$ is the recruitment rate of the population; $\mu_S,\mu_I,\mu_R>0$ are the death rates of the susceptible, infected and recovered individuals, respectively; $r>0$ is the recovery rate of the infected individuals and $F(S(t),I(t))$ is the incidence rate. To simplify the study, it has been noted that
the dynamics of recovered individuals have no effect on the disease transmission dynamics. Thus, following the usual practice,
the recovered individuals are removed from the formulation henceforth.
 Some kinds of the incidence rates are considered such as
\begin{itemize}
\item The Holling type II functional response (see e.g., \cite{HUO}): $F(S,I)=\frac{\beta SI}{m_1+S}\cdot$
\item The bilinear functional response (see e.g., \cite{DANG,ZHANG}): $F(S,I)=\beta SI$.
\item The nonlinear functional response (see e.g., \cite{RUAN,YANG}): $F(S,I)=\dfrac {\beta SI^l}{1+m_2I^h}\cdot$
\item The Beddington-DeAngelis functional response (see e.g., \cite {NHU1, NHU2}): 
$F(S,I)=\dfrac{\beta SI}{1+m_1S+m_2I}\cdot$
\end{itemize}
For the deterministic SIR models with these incidence rates, the researchers have found the reproduction number $R_0$ which has the property: if $R_0<1$ then the disease free equilibrium point is locally asymptotically stale; in case $R_0>1$ we see that the disease point is unstable and there is a steady state, which is locally asymptotically stable. 

However, it is well recognized that the environment is often affected by some random factors such as the temperature, the climate, the water resources, etc. Thus, it is important to consider the stochastic epidemic models. By these random effects, the death rates and the incidence rate are often perturbed by white noises. Many authors have considered the stochastic SIR models when the natural death rates are affected by white noises, i.e., $\mu_S \hookrightarrow \mu_S + \sigma_1\dot B_1(t), \mu_I+r\hookrightarrow \mu_I+r+\sigma_2 \dot B_2(t)$, where $B_i(t),i=1,2$ are Brownian motions and the stochastic equation in general has the form (see e.g., \cite{DANG,NHU2,YANG})
\begin{equation*}
\begin{cases}
dS(t)=\big[a_1-b_1S(t)-I(t)f(S(t),I(t))\big]dt + \sigma_1 S(t) dB_1(t),\\
dI(t)=\big[-b_2I(t) + I(t)f(S(t),I(t))\big]dt + \sigma_2I(t) dB_2(t),\\
\end{cases}
\end{equation*}
where,
we have rewritten the coefficients: $b_1=\mu_S, b_2=\mu_I+r$ and $F(S,I)=If(S,I)$ (compared with \eqref{1.0}).
In an other motivation, some authors have studied the models, in which, the white noise acts on  some special incidence functional responses, i.e., $f(s,i)\hookrightarrow f(s,i)+g(s,i)\dot B_3(t)$ and the stochastic equation becomes (see e.g., \cite{CAI})
\begin{equation*}
\begin{cases}
dS(t)=\big[a_1-b_1S(t)-I(t)f(S(t),I(t))\big]dt -I(t)g(S(t),I(t))dB_3(t),\\
dI(t)=\big[-b_2I(t) + I(t)f(S(t),I(t))\big]dt + I(t)g(S(t),I(t))dB_3(t).\\
\end{cases}
\end{equation*}
 By these motivations, the main aim of this paper is to generalize this problem by two ways. We study the stochastic SIR equation with more general incidence functional responses and in which, both natural death rates and the incidence rate are perturbed by white noises. Precisely, we consider the stochastic SIR equation as following
\begin{equation}\label{e2.1}
\begin{cases}
dS(t)=\big[a_1-b_1S(t)-I(t)f(S(t),I(t))\big]dt + \sigma_1 S(t) dB_1(t) -I(t)g(S(t),I(t))dB_3(t),\\
dI(t)=\big[-b_2I(t) + I(t)f(S(t),I(t))\big]dt + \sigma_2I(t) dB_2(t) + I(t)g(S(t),I(t))dB_3(t),
\end{cases}
\end{equation}
and provide a threshold number $R$ for the stochastic epidemic SIR model \eqref{e2.1} that has the same properties as the reproduction number $R_0$ in the deterministic case. This means that when $R<1$ the number of the infected individuals $I(t)$ tends to zero with the exponential rate while the number of the susceptible individuals $S(t)$ converges exponentially to the solution on the boundary. In case of $R>1$, the solution has a unique invariant measure concentrated on $\R_+^{2,\circ}$ and the transition probability converges to the invariant measure in total variation norm with a polynomial of any degree rate. The ergodic property is also obtained in this case. 

One of the main difficulties in studying this model is that  the comparison theorem \cite[Theorem 1.1, p.437]{IW} to compare the solution of \eqref{e2.1} with the solution on boundary as in \cite{DANG,NHU2} is no longer valid because there are complex white noises attended in the stochastic equation \eqref{e2.1}. Therefore, we can not approach the problem as usual and some new techniques must require here.

The rest of the paper is arranged as follows. Section \ref{secpre} provides some preliminary results about the system and introduce the threshold $R$ to determine the permanence and extinction of the system. In Section \ref{secext}, we derive the condition for the extinction of the system \eqref{e2.1}, which is equivalent to the case $R<1$ while section \ref{secper} focuses on the condition for permanence, corresponding to the case $R>1$. The last section is devoted to some numerical examples as well as discussing the obtained results in this paper.
\section {Preliminary results and the threshold $R$}\label{secpre}
Throughout of this paper, we assume that the incidence rate 
and  the diffusion term 
 satisfy the following conditions.
 \begin{asp}\label{conditionfg}\
Assume that
\begin{itemize}
\item $f,g$ are non-negative functions and $f(0,i)=0,  g(0,i)=0\;\forall i\geq 0$,
\item there exist   positive constants $F, G, K$ such that
\begin{equation*}
\begin{aligned}\abs{f(s_1,i_1)-f(s_2,i_2)}&\leq F (\abs{s_1-s_2}+\abs{i_1-i_2}),\\
\abs{g(s_1,i_1)-g(s_2,i_2)}&\leq G (\abs{s_1-s_2}+\abs{i_1-i_2}),\\
\abs{ig(s_1,i)-ig(s_2,i)}&\leq G\abs{s_1-s_2}\quad \text{and}\quad g(s,i)\leq K,
\end{aligned}
\end{equation*}
for all $s_1,s_2,i_1,i_2,s,i\geq 0$.
\end{itemize}
\end{asp}
We note that the bilinear incidence rate, the Beddington-DeAngelis incidence rate, the Holling type II functional response are  special cases of this general incidence function.
\subsection{The existence and uniqueness of the solution}
Let $(\Omega,\F,\{\F_t\}_{t\geq0},\PP)$ be a complete probability space with the filtration $\{\F_t\}_{t\geq0}$ satisfying the usual conditions and $B_1(t),B_2(t),$ $B_3(t)$ be mutually independent Brownian motions.
\begin{thm}
For any initial point $(u,v)\in \R_+^{2,\circ}:=\{(u',v')\in \R^2: u',v'>0\}$, there exists a unique global solution $(S_{u,v}(t), I_{u,v}(t)), t\geq0$ of \eqref{e2.1} with initial value $(S_{u,v}(0),I_{u,v}(0))=(u,v)$. Further, $(S_{u,v}(t),I_{u,v}(t))\in \R_+^{2,\circ}$ for all $t>0$.
\end{thm}
\begin{proof}
It is noted that although we have assumed $f(s,i)$ is Lipschitz continuous, the coefficient $if(s,i)$ in the system \eqref{e2.1} is non-Lipschitz in general. Since the coefficients of the equation are locally Lipschitz continuous, there is a unique solution $(S_{u,v}(t),I_{u,v}(t))\in \R_+^{2,\circ}$ with the  initial value $(u,v)\in \R_+^{2,\circ}$, defined on maximal interval  $t\in [0,\tau_e)$. 
We need to show $\tau_e=\infty$ a.s.
Let us consider  the Lyapunov function $V: \R_+^{2}\rightarrow \R_+$ 
$$V(s,i)=\left(s-C_1-C_1\ln{\frac {s}{C_1}}\right)+(i-1-\ln i)\;\;\;\;\;\text{where}\;\;\;\;\;C_1:=\frac {b_2}F\cdot$$
By directly calculating the differential operator $\Lom V(s,i)$, we have
\begin{equation*}
\begin{aligned}
\Lom V(s,i)&=\left(1-\frac {C_1}{s}\right)\left(a_1-b_1s-if(s,i)\right)+ \frac {C_1}{2s^2}\left(\sigma_1^2s^2+i^2g^2(s,i)\right)
\\
&\;\;\;+\left(1-\frac 1{i}\right)\left(-b_2i+if(s,i)\right) +\frac 1{2i^2}\left(\sigma_2^2i^2+i^2g^2(s,i)\right).
\end{aligned}
\end{equation*}
It follows from the assumption \ref{conditionfg} that
$$f(s,i)=\abs{f(s,i)-f(0,i)}\leq Fs\;\;\;\;\;\text{and}\;\;\;\;\;ig(s,i)=\abs{ig(s,i)-ig(0,i)}\leq Gs.$$
Therefore, it is easily seen that
\begin{equation*}
\begin{aligned}
\Lom V(s,i)
&\leq C_2+\frac{C_1f(s,i)i}{s}+\frac{C_1 g^2(s,i)i^2}{2s^2}+\frac{g^2(s,i)}2-b_2i
\\& \leq C_2+ \frac{C_1 G^2}2+\frac{K^2}2 +(C_1 F-b_2)i
\\& = C_2+ \frac{C_1 G^2}2+\frac{K^2}2,\quad \text{where}\;C_2= a_1+C_1 b_1+\frac{C_1\sigma_1^2}2+b_2+\frac{\sigma_2^2}2.
\end{aligned}
\end{equation*}
As a result, $\Lom V(s,i)$ is bounded in $\R_+^{2,\circ}$. By using the same argument in the proofs in \cite[Theorem 2.1, p. 994]{LAH}, we complete the proof of the theorem.
\end{proof}
\subsection{Preliminary estimates about the expectation} 
Via  Lyapunov functions we estimate moments of $S_{u,v}(t),I_{u,v}(t)$ 
that are shown  in the following Lemma.
\begin{lem}\label{lem2.1}
 The following assertions hold:
\begin{itemize}
\item[{\rm(i)}] For any $0<p<\min\left\{\dfrac{2b_1}{\sigma_1^2},\dfrac{2b_2}{\sigma_2^2}\right\}$ and $\bar p>0$, there is a constants $Q_1$ such that
$$\limsup\limits_{t\to\infty}\E \left[(S_{u,v}(t)+ I_{u,v}(t))^{1+p}+(S_{u,v}(t)+ I_{u,v}(t))^{-\bar p}\right]\leq Q_1\;\; \forall (u,v) \in \R_+^2.$$
\item[{\rm(ii)}] For any $\eps>0$, $H>1$, $T>0$, there is $\bar H =\bar H(\eps, H, T)$ such that 
$$\PP \left\{\dfrac{1}{\bar H}  \leq S_{u,v}(t) 
 \leq \bar H \;\;\forall t \in [0,T] \right\}\geq 1-\eps \;\; \text{if}\;\;\; (u,v) \in [H^{-1},H]\times [0;H],$$
and
$$\PP\{ 0 \leq S_{u,v}(t), I_{u,v}(t) \leq \bar H \;\;\forall t \in [0,T] \}\geq 1-\eps \;\; \text{if}\;\;\; (u,v) \in [0,H]\times [0;H].$$
\end{itemize}
\end{lem}
\begin{proof}
Consider Lyapunov function
$V_1(s,i)=(s+i)^{1+p}+(s+i)^{-\bar p}
.$ By directly calculating the differential operator $\Lom V_1(s,i)$, we obtain
\begin{align*}
\Lom V_1(s,i)=&(1+p)(s+i)^p(a_1-b_1s-b_2i)+\frac {p(1+p)}2(s+i)^{p-1}\Big( \sigma_1^2s^2+\sigma_2^2i^2\Big)\\
&-\bar p(s+i)^{-\bar p-1}(a_1-b_1s-b_2i)+\frac {\bar p(1+\bar p)}2(s+i)^{-\bar p-2}\Big( \sigma_1^2s^2+\sigma_2^2i^2\Big)\\
\leq &(1+p)a_1(s+i)^p-(1+p)(s+i)^{p-1}\Big[ (b_1-\dfrac p2 \sigma_1^2)s^2+(b_2-\dfrac p2 \sigma_2^2)i^2+ (b_1+b_2)si\big] \\
&-\bar pa_1(s+i)^{-\bar p-1}+\bar p\left[\max\{b_1,b_2\}+\frac {1+\bar p}2\max\{\sigma_1^2,\sigma_2^2\}\right](s+i)^{-\bar p}.
\end{align*}
Let 
$0<C_3<(1+p)\min\Big\{b_1-\dfrac p2 \sigma_1^2,b_2-\dfrac p2 \sigma_2^2, \dfrac{b_1+b_2}2\Big\}$. By some standard calculations, we get
$$C_4=\sup_{(s,i)\in \R^2_+\setminus \{(0,0)\}}\left\{\Lom V_1(s,i)+C_3V_1(s,i)\right\}<\infty.$$
That means
\begin{equation}\label{10}
\Lom V_1(s,i) \leq C_4-C_3 V_1(s,i).
\end{equation}
Applying \cite[Theorem 5.2, p.157]{MAO} obtains the part (i) of  Lemma.

Now, we move to the proof of the part (ii). By \eqref{10}, 
 there exist $h_1$,$h_2>0$ (see \cite[Lemma 2.1, p. 45]{DA}) such that for all $(u,v)\in [0,H]\times [0,H]$
\begin{equation*}
\mathbb{P}  \Big\{0 \leq S_{u,v} (t) \leq h_1 \;\forall t\in [0,T] \Big\} \geq 1-\frac\varepsilon2\;\text{and}\;\mathbb{P}  \Big\{0 \leq I_{u,v} (t) \leq h_2 \;\forall t\in [0,T] \Big\} \geq 1-\dfrac{\varepsilon}{2}.
\end{equation*}
Let 
$$\Omega'_1=\Big\{0 \leq I_{u,v} (t) \leq h_2\;\forall t\in [0,T]\Big\},\;\;\Omega'_1\;\text{may depend on}\;(u,v).$$
 By exponential martingale inequality \cite[Theorem 7.4, p. 44]{MAO} we have $\PP(\Omega'_2)\geq 1-\dfrac{\eps}{2}$, where 
\begin{align*}
\Omega'_2&=\left\{ -\sigma_1 B_1(t) + \int_0^t\dfrac{I_{u,v}(s)g(S_{u,v}(s),I_{u,v}(s))}{S_{u,v}(s)}dB_3(s)\right.\\
 &\hskip 0.5cm\leq \left.\dfrac{\sigma_1^2 t}{2}+\dfrac12\int_0^t\dfrac{I^2_{u,v}(s)g^2(S_{u,v}(s),I_{u,v}(s))}{S^2_{u,v}(s)}ds + \ln \dfrac{2}{\eps}\;\;\forall t \geq 0\right\},\;\;\Omega'_2\;\text{may depend on}\;(u,v).
\end{align*}
Applying  It\^{o}'s formula to the equation \eqref{e2.1} yields that
\begin{multline}\label{lns}
\ln S_{u,v}(t)=\ln u +\int_0^t \dfrac{a_1}{S_{u,v}(s)}ds -\left(b_1+\dfrac{\sigma_1^2}{2}\right) t-\int_0^t \dfrac{I_{u,v}(s)f(S_{u,v}(s),I_{u,v}(s))}{S_{u,v}(s)}ds 
\\\;\;\;- \dfrac12\int_0^t \dfrac{I^2_{u,v}(s)g^2(S_{u,v}(s),I_{u,v}(s))}{S^2_{u,v}(s)}ds  +\sigma_1B_1(t) - \int_0^t\dfrac{I_{u,v}(s)g(S_{u,v}(s),I_{u,v}(s))}{S_{u,v}(s)}dB_3(s).
 \end{multline}
For all $(u,v)\in [H^{-1},H]\times[0,H]$, $\omega\in \Omega'_1\cap\Omega'_2$ and $t \in [0,T]$, by the assumption \ref{conditionfg} we have
\begin{equation*}
\begin{aligned}
\ln S_{u,v}(t)&\geq \ln u -(b_1+\sigma_1^2) t-F\int_0^t I_{u,v}(s)ds -G^2\int_0^t I^2_{u,v}(s)ds-\ln \dfrac 2{\eps}
\\&\geq \ln H^{-1}-(b_1+\sigma_1^2)T-h_2FT-h_2G^2T-\ln \dfrac 2{\eps}:=\ln h_3^{-1}
\end{aligned}
\end{equation*}
By setting $\bar H=\max\{1, h_1, h_2, h_3\}$ we complete the proof.
\end{proof}

Moreover, we note that  $I_{u,v}(t)=0\;\forall t\geq 0$\a.s provided $I_{u,v}(0)=0$.
Further, it follows \cite[Theorem 2.9.3]{MAO} and \cite[Section 2.5]{YIN} that the solution of \eqref{e2.1} is homogeneous strong Markov and Feller process if provided that the coefficients are global Lipschitz.
Therefore, by using the results in part (ii) of Lemma \ref{lem2.1}, we obtain from the local Lipschitz property of coefficients of \eqref{e2.1} and a truncated argument that $(S_{u,v}(t),I_{u,v}(t))$ is homogeneous strong Markov and Feller process.
The details of this truncated argument and this result can be found in \cite[Theorem 5.1]{NYZ17}.
\subsection{The threshold $R$}
Consider the equation on boundary when the infected individuals are absent, i.e.,
\begin{equation}\label{phi}
d \varphi(t)=\big(a_1-b_1\varphi(t)\big)dt+\sigma_1 \varphi(t)dB_1(t),\;\;\;\varphi(0)\geq 0.
\end{equation}
We write $\varphi_u(t)$ for the solution of the equation \eqref{phi} with the initial condition $\varphi(0)=u$. 
By solving the Fokker-Planck equation,  the equation \eqref{phi}  has
a unique stationary distribution with density $f^*$ given by
\begin{equation}\label{hpp}
f^*(x)=\dfrac{b^a}{\Gamma(a)}x^{-(a+1)}e^{-\frac{b}x},\; x>0,
\end{equation}
where $c_1=b_1+\dfrac{\sigma_1^2}{2}, a=\dfrac{2c_1}{\sigma_1^2},b=\dfrac{2a_1}{\sigma_1^2} $ and
$\Gamma(\cdot) $ is Gamma function. Our idea is to determine whether $I_{u,v}(t)$ converges to 0 or not by considering the Lyapunov exponent $\limsup_{t\to\infty} \dfrac{\ln I_{u,v}(t)}t$ when $I_{u,v}(t)$ is small. Using It\^{o}'s formula gets
\begin{equation}\label{Lia}
\begin{aligned}
\dfrac{\ln I_{u,v}(t)}{t}=&\dfrac{\ln v}{t}+ \dfrac {\sigma_2B_2(t)}t+\dfrac 1{t}\int_0^tg(S_{u,v}(s),I_{u,v}(s))dB_3(s)
\\&-c_2+\dfrac 1{t}\int_0^t\Big(f(S_{u,v}(s),I_{u,v}(s))-\dfrac 12 g^2(S_{u,v}(s),I_{u,v}(s))\Big)ds,
\end{aligned}
\end{equation}
where $c_2=b_2+\dfrac{\sigma_2^2}2.$
Intuitively, $\limsup_{t\to\infty} \frac{\ln I_{u,v}(t)}t<0$ implies $\lim_{t\to\infty} I_{u,v}(t)=0$ and when $I_{u,v}(t)$ is small then $S_{u,v}(t)$ is close to $\varphi_u(t)$ and therefore, when $t$ is sufficiently large we have
\begin{equation*}
\begin{aligned}
&\dfrac 1{t}\int_0^t\Big(f(S_{u,v}(s),I_{u,v}(s))-\dfrac 12 g^2(S_{u,v}(s),I_{u,v}(s))\Big)ds
\approx \dfrac 1{t}\int_0^t\Big(f(\varphi_{u}(s),0)-\dfrac 12 g^2(\varphi_u(s),0)\Big)ds.
\end{aligned}
\end{equation*}
By boundedness of $g(\cdot,\cdot)$;  strong law of large numbers \cite[Theorem 3.16, p.46]{SKO} for $\varphi_{u}(t),$ from \eqref{Lia} we obtain that the Lyapunov exponent of $I_{u,v}(t)$ is approximated to
\begin{equation}\label{ld}
-c_2+ \int_0^{\infty}\Big(f(x,0) - \dfrac 12 g^2(x,0) \Big)f^*(x)dx:=\lambda.
\end{equation}
Roughly speaking, if $\lambda >0$, whenever $I_{u,v}(t)$ is enough small, $\limsup_{t\to\infty}\frac{\ln I_{u,v}(t)}{t}\approx \lambda>0$ and it leads to $I_{u,v}(t)$ can not be very small in a long time. Conversely, when $\lambda <0$, if the solution starts from a initial point $(u,v)$, where v is sufficiently small then $\limsup_{t\to\infty}\frac{\ln I_{u,v}(t)}{t}\approx \lambda<0$ and which implies $I_{u,v}(t)\to 0$. Therefore, the remaining work is to investigate how the solution enters the region $\{(u,v): v\;\text{is sufficient small}\}$. However, that is just in intuition, the detailed proofs are very technical and complex and need to be carefully done.
\begin{rem}
 For $\lambda$ defined as in \eqref{ld}, $\lambda<0$ is equivalent to 
\begin{equation}\label{r}
R:=\dfrac{\int_0^{\infty}f(x,0)f^*(x)dx}{c_2+\int_0^{\infty}\frac 12 g^2(x,0)f^*(x)dx}<1.
\end{equation}
Therefore, we expect $R$ to be a threshold between the persistence and extinction of \eqref{e2.1} as in the deterministic case. 
It is noted that since the conditions on $f(s,i)$, it is easy to see that $f(s,i)\leq Fs$ and thus, $$\int_0^{\infty}\Big(f(x,0) - \dfrac 12 g^2(x,0) \Big)f^*(x)dx\leq F\int_0^{\infty}xf^*(x)dx=\dfrac{Fa_1}{b_1}.$$
As a result, $\lambda, R$ are well-defined.
\end{rem}
\section{Extinction}\label{secext}
Consider the case  
$R<1$ or equivalently, $\lambda<0$. 
We shall show that the number of the infected individuals $I_{u,v}(t)$ tends to zero with the exponential rate while the number of the susceptible individuals $S_{u,v}(t)$ converges to $\varphi_u(t)$.  
The problem here is that  we can not apply the comparison theorem \cite[Theorem 1.1, p.437]{IW} for $S_{u,v}(t)$ and $\varphi_u(t)$ to use  a similar argument as in \cite{DANG,NHU2}. 
\begin{thm}\label{R<1}
Assume that $\lambda<0$. We also assume that the function $f(s,0)-\frac 12 g^2(s,0)$ is monotonic. Then for any initial point $(u,v)\in\mathbb{R}_+^{2,\circ}$, 
the number of the infected individuals $I_{u,v}(t)$ tends to zero  with the exponential rate $\lambda$ and the susceptible class $S_{u,v}(t)$ converges exponentially to the solution on boundary $\varphi_u(t)$. Precisely, 
$$\limsup\limits_{t\to\infty}\dfrac{\ln I_{u,v}(t)}{t}=\lambda\a.s$$
and
$$\limsup\limits_{t\to\infty}\dfrac{\ln \abs{S_{u,v}(t)-\varphi_u(t)}}{t}\leq \max\{\lambda,-c_1\}\a.s$$

\end{thm}
In order to prove Theorem \ref{R<1}, we need some following auxiliary results.
\begin{prop}\label{lem2.2}
For any $T,H>1,$ $\eps >0, \theta >0$, there is a $\delta =\delta (H,T,\eps,\theta)$ such that
$$\PP\{ \tau_{u,v}^{\theta} \geq T \}\geq 1-\eps\;\;\forall \; (u,v) \in [0,H]\times (0,\delta],$$
where $\tau_{u,v}^{\theta}=\inf\{t\geq 0: I_{u,v}(t)> \theta \}.$
\end{prop}
\begin{proof}
By exponential martingale inequality \cite[Theorem 7.4, p. 44]{MAO}, we have $\PP(\Omega''_1)\geq 1-\dfrac{\eps}{2}$, where
$$\Omega''_1=\left\{ \sigma_2 B_2(t) + \int_0^tg(S_{u,v}(s),I_{u,v}(s))dB_3(s) \leq \dfrac{\sigma_2^2t}{2}+\dfrac12\int_0^tg^2(S_{u,v}(s),I_{u,v}(s))ds + \ln \dfrac{2}{\eps}\;\forall t\geq 0\right\},$$
$\Omega''_1$ may depend on $(u,v)$. In view of part (ii) Lemma \ref{lem2.1}, there exists $\bar H=\bar H(T,H,\eps)$ such that $\PP(\Omega''_2)\geq 1-\dfrac{\eps}{2}$, where
$$\Omega''_2=\{ 0 \leq S_{u,v}(t),I_{u,v}(t) \leq \bar H \;\;\forall t \in [0,T] \},\;\;\Omega''_2\;\text{may depend on}\;(u,v).$$
Applying  It\^{o}'s formula to the equation \eqref{e2.1} implies that
\begin{multline}\label{lni}
\ln I_{u,v}(t)=\ln v -c_2t+\int_0^tf(S_{u,v}(s),I_{u,v}(s))ds - \dfrac12\int_0^tg^2(S_{u,v}(s),I_{u,v}(s))ds\\ 
+\sigma_2B_2(t) + \int_0^tg(S_{u,v}(s),I_{u,v}(s))dB_3(s).
 \end{multline}
Therefore, for any $(u,v) \in [0,H]\times (0,H]$ and $\omega \in \Omega''_1 \cap \Omega''_2$ we have
$$\ln I_{u,v}(t) < \ln v - b_2t+t\big(2F\bar H+f(\bar H, \bar H)\big) + \ln \dfrac{2}{\eps}\;\;\forall t\in [0,T]. $$
Hence, we can choose a sufficiently small $\delta =\delta (H,T,\eps,\theta)<H$  so that for all $(u,v)\in [0,H]\times (0,\delta]$ and $0\leq t\leq T$, $\ln I_{u,v}(t)<\ln\theta$. The proof is complete.
\end{proof}
\begin{prop}\label{extinction}
For any $H,T>1$, $\eps, \nu>0$, there exists $\theta>0$ such that for all $(u,v) \in [0,H]\times(0,\theta],$
$$\PP\big\{\abs{S_{u,v}(t)-\varphi_u(t)}\leq \nu\;\;\forall\; 0\leq t \leq T\wedge \tau_{u,v}^{\theta}\big\}\geq 1-\eps.$$
\end{prop}
\begin{proof}
First, in view of part (ii) in Lemma \ref{lem2.1}, there exists $\bar H$ such that 
\begin{equation}\label{2018-1}
\PP\{0\leq S_{u,v}(t),\varphi_u(t)\leq \bar H\;\forall t\in [0,T]\}\geq 1-\dfrac {\eps}2\;\;\forall (u,v)\in [0,H]\times [0,H].
\end{equation}
Second, by a same argument  as in the proofs of \cite[Lemma 3.2]{DANG1} or \cite[Lemma 6.9]{MAO}, there exists a constant $C=C(H,\bar H,T,\theta )$  satisfying
\begin{equation*}
\begin{cases}
\E\sup\limits_{0\leq t\leq T\wedge \tau_{u,v}^{\theta}\wedge \inf\{t\geq 0: S_{u,v}(t)\vee \varphi_u(t)>\bar H\}}\big(S_{u,v}(t)-\varphi_u(t)\big)^2\leq C,\\
C\to 0\;\text{ when}\; \theta \to 0.
\end{cases}
\end{equation*}
Therefore, by virtue of 
Chebyshev's inequality and \eqref{2018-1}, we can choose a sufficiently small constant  $\theta$  such that
\begin{equation*}
\PP\{\abs{S_{u,v}(t)-\varphi_u(t)}\leq \nu\;\;\forall 0\leq t \leq T\wedge\tau_{u,v}^{\theta}\}\geq 1-\eps,\;\;\forall (u,v) \in [0,H]\times(0,\theta].
\end{equation*}
\vskip -10mm
\end{proof}

The following lemma is a generalization of the law of iterated logarithm.
\begin{lem} \label{br}
Let $W(t)$ be a standard Brownian motion and $\phi_t$ be a
stochastic process, $\mathcal  F_t-$ progressively measurable such
that
$$ P\left\{ \int_0^t \phi_s^2\, ds < \infty\right \} = 1 \quad
\text{ for all} \quad  t>0.$$
Then for any $\eps>0$ there exists a constant $q_\eps$, independent of process $\phi$  such that
$$\PP\left\{\bigg |\int_0^t \phi_s\, dW_s \bigg | \leq q_\eps \sqrt
{ m(t) \ln (\abs{m(t)} +1) }\right\}\geq 1-\eps,$$
 where $m(t) =
\int_0^t \phi_s^2 \,ds$.
\end{lem}
\begin{proof} For simplifying  notations, we set
$$ M(t) = \int_0^t\phi_s\,dW_s\;;\; m(t) = \int_0^t
\phi_s^2 \, ds.$$
We define a family of stopping times $\tau_1(t)$ given
by
$$ \tau_1(t)  = \begin{cases} \inf \{\; s\geq 0: m(s) > t \} \\
             \infty \;\;\text { if } \;\; t \geq m(\infty) =
\lim_{ t\uparrow \infty} m(t).\end{cases} $$ 
Applying \cite[Theorem 7.2, p.92]{IW}, on an extension  $(\overset
{\sim} \Omega, \overset {\sim} {\mathcal  F}, \overset
{\sim}  \PP)$ of $(\Omega, \mathcal  F,\PP)$, there exists an
$\overset {\sim}
{\mathcal  F} - $ Brownian motion $\mu(t)$
such that $\mu(t) = M(\tau_1(t)),$ $ t \in [0, \infty).$
Consequently, we can represent $M(t)$ by an $\overset {\sim}
{\mathcal  F} - $ Brownian motion $\mu(t)$ and the stopping times
$\tau_1(m(t))=t$, i.e.,
$$\int_0^t\phi_s\,dW_s = \mu(m(t)).$$
On the other hand,
by virtues of the law of iterated
logarithm we have
$$\limsup_{ t \to \infty} \frac {|\mu(t)|}{
\sqrt{2t\ln|\ln t|}} = 1; \quad \limsup_{ t \to 0} \frac
{|\mu(t)|} { \sqrt{2t\ln|\ln t|}} =1 \quad \text {
a.s.}$$
Therefore, the random variable $\Phi$ defined by
$$\Phi := \sup_{0< t< \infty} \frac {|\mu(t)|} {
\sqrt{t\big(|\ln t| + 1\big)}}$$
is finite\a.s, i.e., $\PP\{ \Phi < \infty\}=1$ and the
distribution of $\Phi$ does not depend on the process $(\phi_t)$.
The definition of $\Phi$ implies that
$$|\mu(t)| \leq \Phi\cdot  \sqrt{t\big(|\ln t| +1\big)}.$$
Hence, one has
$$\bigg| \int_0^t \phi_s(\omega)\, dW_s \bigg| \leq
\Phi \cdot \sqrt{ m(t)\big(|\ln m(t)| +1\big)}.$$ 
Since $\Phi$ is finite \a.s and the distribution of $\Phi$ does not depend on $\phi$, for any $\eps>0$ there exists $q_\eps$ independent of $\phi$ such that $\PP\{\Phi<q_\eps\}\geq 1-\eps.$ Lemma \ref{br} is proved.
\end{proof}

\begin{lem}\label{phi1} 
For any $u\geq 0$, one has
$$\lim_{t\to\infty}\frac{\varphi_u(t)}t=0\a.s$$
\end{lem}
\begin{proof} 
It is seen that as in Lemma \ref{lem2.1}, for $p<\frac{2b_1}{\sigma_1^2}$, we have
$ \mathcal Ls^{\frac {1+p}2}\leq C_4-C_3s^{\frac {1+p}2}$, for some constants $C_4,C_3$.
Using this fact and It\^o's formula for $\varphi_u^{\frac {1+p}2}(t)$, we get that for all $n\in\N$
\begin{equation*}
\sup_{t\in [n,n+1]}\varphi_u^{\frac{1+p}2}(t)\leq\varphi_u^{\frac{1+p}2}(n)+C_4+\sup_{t\in[n,n+1]}\int _n^{t}\frac{1+p}2\varphi_u^{\frac{1+p}2}(s)dB_1(s).
\end{equation*}
As a consequence, one has
\begin{equation}\label{e2.6bs}
\E\sup_{t\in [n,n+1]}\varphi_u^{1+p}(t)\leq2(\varphi_u^{\frac{1+p}2}(n)+C_4)^2+2\E\left(\sup_{t\in[n,n+1]}\int _n^{t}\frac{1+p}2\varphi_u^{\frac{1+p}2}(s)dB_1(s)\right)^2.
\end{equation}
It is similar to Lemma \ref{lem2.1} to obtain that
$$
\sup_{t\in[0,\infty]}\E \left(\varphi_u^{1+p}(t)+\varphi_u^{\frac{1+p}2}(t)\right)\leq \bar Q_1,
$$
for some finite constant $\bar Q_1$.
Which together with \eqref{e2.6bs} and Burkholder-Davis-Gundy inequality give us that
\begin{equation}\label{e2.6bs1}
\E \sup_{t\in [n,n+1]}\varphi_u^{1+p}(t)\leq 2\left(\bar Q_1+C_4\right)^2+8\bar Q_1=:\hat Q.
\end{equation}
 For any $\eps>0$,
put $A_n=\Big\{\dfrac {\sup_{n\leq t\leq n+1}\varphi_u(t)}{n}\geq \eps\Big\}$. From  \eqref{e2.6bs1}, we have $\PP(A_n)\leq (\eps n)^{-(1+p)}\hat Q$, which implies 
$\sum_{n=1}^\infty \PP(A_n)<\infty$. Hence,  $\PP\{A_n \text{ i.o.}\}=0$ by Borel-Cantelli lemma. In other word, $ \limsup\limits_{n\to\infty}\dfrac {\sup_{n\leq t\leq n+1}\varphi_u(t)}{n}\leq \eps$ a.s. Since $\eps$ is arbitrary and $\varphi_u(t)>0$,
the 
 Lemma is proved. 
\end{proof}

\begin{prop}\label{lm3.1}
Assume that the assumption in Theorem \ref{R<1} holds. For any $0<\varepsilon<\min\{\frac 19,-\frac{\lambda}{9}\}$ and $H>1$, there exists $\widehat \delta=\widehat \delta(\eps, H)\in (0,H^{-1})$ such that
$$\mathbb{P}\left\{\limsup_{t\to\infty}\abs{\dfrac{\ln I_{u,v}(t)}{t}- \lambda}\leq \eps  \right\}\geq 1-8\eps\;\forall (u,v)\in [H^{-1};H]\times(0;\widehat \delta].$$
\end{prop}
\begin{proof}
First, we consider the case the function $f(s,0)-\frac 12 g^2(s,0)$ is non-decreasing.
In what follows, although some sets $\Omega_i^{u,v}$ may not depend on $(u,v)$, we still use the superscript $(u,v)$ for the consistence of notations.
Our idea in this proposition is to estimate simultaneously  $\ln I_{u,v}(t)$ and the difference $\abs{S_{u,v}(t)-\varphi_u(t)}$. To start, we need some following primary estimates. 
By definition of $\lambda$ and ergodicity of $\varphi_{H}(t)$ we obtain
\begin{equation*}
-c_2+ \lim_{t\to\infty}\dfrac1t\int_0^{t}f(\varphi_{H}(s),0)ds-\lim_{t\to\infty}\dfrac{1}{2t}\int_0^{t}g^2(\varphi_{H}(s),0)ds = \lambda \a.s
\end{equation*}
So, there exists $T_1=T_1(\eps)>1$ such that $\PP(\Omega_1^{u,v})\geq 1-\eps$, where 
$$\Omega_1^{u,v}=\left\{-c_2+\dfrac1t \int_0^t f(\varphi_{H}(s),0)ds-\dfrac1{2t} \int_0^tg^2(\varphi_{H}(s),0)ds\leq \lambda + \eps,\;\forall t\geq T_1\right\} .$$
Furthermore, by exponential martingale inequality 
we have $\mathbb{P}(\Omega_2^{u,v})\geq1-\eps$, where
\begin{align*}
\Omega_2^{u,v}&=\left\{\int_0^t g(S_{u,v}(s),I_{u,v}(s))dB_3(s)
\leq \dfrac{K^2}{\eps}\ln \dfrac{1}{\varepsilon} + \dfrac{\eps}{2K^2}\int_0^t g^2(S_{u,v}(s),I_{u,v}(s))ds,\;\forall t\geq 0\right \}.
\end{align*}
Since $\lim\limits_{t\to\infty}\dfrac{B_2(t)}{t}=0 \text{\;a.s.\;} $, there is $T_2=T_2(\eps)>1$ such that $\mathbb{P}(\Omega_3^{u,v})\geq1-\eps$ where
$$\Omega_3^{u,v}=\left\{\sigma_2\dfrac{B_2(t)}{t}<\dfrac{\eps}2\;\forall t \geq T_2\right\}.$$
On the other hand, by Lipschitz continuity of $f,g$ and boundedness of $g$, we can choose $0<\nu=\nu(\eps) <  \min\Big\{1,\dfrac{\eps}{2(F+KG)}\Big\}$ such that 
$$
\left\vert f(s_1,i_1)-f(s_2,i_2)\right\vert <\eps\;\;\;\text{and}\;\;\;
\left\vert \dfrac{1}2g^2(s_1,i_1)-\dfrac{1}2 g^2(s_2,i_2)\right\vert <\eps, $$
provided $\left\vert s_1-s_2\right\vert < \nu$ and $\left\vert i_1-i_2\right\vert <\nu$.

By Lemma \ref{br}, there exists $q_\eps$, independent of $u,v$
such that $\PP(\Omega_4^{u,v})\geq 1-\eps$, where
\begin{equation*}
\begin{aligned}
\Omega_4^{u,v}:=&\Big\{\abs{\sigma_1B(t)}\leq q_\eps \sqrt{t(\abs{\ln t}+1)}\;\forall t\geq 0\Big\}
\\&\bigcap \Bigg\{\abs{\int_T^t e^{c_1s-\sigma_1B_1(s)}I_{u,v}(s)g(S_{u,v}(s),I_{u,v}(s))dB_3(s)}\leq q_\eps \sqrt {n(t)(\abs{\ln n(t)}+1)}\;\forall t\geq T\Bigg\},
\end{aligned}
\end{equation*}
and $T:= T_1\vee T_2$ and $n(t)=\int_T^t e^{2c_1s+2q_\eps \sqrt{s(\abs{\ln s} +1)}}I^2_{u,v}(s)g^2(S_{u,v}(s),I_{u,v}(s))ds$.
To simplify notations we denote $q_\eps(t):=q_\eps \sqrt{t(\abs{\ln t} +1)}$. 
It is clear that 
$$\Phi_1(\eps):=\sup_{t\geq 0}e^{-c_1t+q_\eps (t)}e^{c_1T+q_\eps(T)}\Big(1+\frac{K^2}{\eps}\ln \frac 1{\eps}+\frac{\eps}2\Big)T<\infty.$$
On the other hand, by Lemma \ref{lem2.1} (ii), there exists $\bar H$ such that $\forall (u,v)\in [0,H]^2$, $\PP(\Omega_5^{u,v})\geq 1-\eps$ where
$$\Omega_5^{u,v}=\big\{0\leq S_{u,v}(t),I_{u,v}(t)\leq \bar H\;\forall t\in [0,T]\big\}.$$
In view of Proposition \ref{extinction}, there exists $\eta>0$ satisfying $$\eta <\min\Big\{\nu, \dfrac{\nu}{2\Phi_1(\eps)\big(2F\bar H+1\big)},\dfrac{\eps}{2(F+KG)}\Big\}$$ such that $\forall (u,v)\in[H^{-1};H]\times(0;\eta],$ $\mathbb{P}(\Omega_6^{u,v})\geq 1-\eps$ where
$$\Omega_6^{u,v}=\left\{\left\vert S_{u,v}(t)-\varphi_u(t)\right\vert  < \nu \; \forall t \leq T \wedge \tau_{u,v}^{\eta}\right\}\;\text{and}\; \tau_{u,v}^\eta=\inf\{t\geq 0: I_{u,v}(t)>\eta\}.$$
We also set $\xi_{u,v}^\nu=\inf\{t\geq 0: \abs{S_{u,v}(t)-\varphi_u(t)}>\nu\}$ and $\zeta_{u,v}:=\xi_{u,v}^\nu \wedge \tau_{u,v}^\eta$. By virtue of Proposition \ref{lem2.2}, there exists $0<\delta<\min\{H^{-1},\eta\}$ such that $\forall (u,v )\in [H^{-1};H]\times(0;\delta]$, $\mathbb{P}(\Omega_7^{u,v}) \geq 1-\eps,$ where
$\Omega_7^{u,v}=\big\{\tau_{u,v}^{\eta} \geq T\big\}.$ Therefore, for all $(u,v) \in [H^{-1};H]\times(0;\delta]$, $\omega \in \cap_{i=1}^{7}\Omega _i^{u,v}$ we have $\zeta_{u,v}\geq T$.

 Now, following the idea introduced at the beginning, we will estimate simultaneously  $\ln I_{u,v}(t)$ and the difference $\abs{S_{u,v}(t)-\varphi_u(t)}$.
It follows from \eqref{lni} that $\forall (u,v) \in [H^{-1},H]\times (0,\delta]$  we have in $\cap_{i=1}^{7}\Omega _i^{u,v}$
\begin{equation}\label{lni2}
\begin{aligned}
&\ln I_{u,v}(t)
=\ln v -c_2t +  \int_0^{t}\left( f(\varphi_H(s),0) -\dfrac{1}{2}g^2(\varphi_{H}(s),0)\right)ds
\\&\;\;\;\;\;+ \int_0^{t}\Big(f(\varphi_u,0)-\dfrac 12 g^2(\varphi_u(s),0)-f(\varphi_H(s),0)+\dfrac 12 g^2(\varphi_H(s),0) \Big)ds
\\&\;\;\;\;\;+\int_0^{t}\Big(f(S_{u,v}(s),0)-f(\varphi_u(s),0) \Big)ds+\dfrac{1}{2}\int_0^{t}\left( g^2(\varphi_u(s),0)-g^2(S_{u,v}(s),0) \right)ds
\\&\;\;\;\;\;+  \int_0^{t}\Big(f(S_{u,v}(s),I_{u,v}(s))-f(S_{u,v},0) \Big)ds+\dfrac{1}{2}\int_0^{t}\left( g^2(S_{u,v}(s),0)-g^2(S_{u,v}(s),I_{u,v}(s))\right)ds
\\&\;\;\;\;\;+\int_0^t g(S_{u,v}(s),I_{u,v}(s))dB_3(s)+ \sigma_2B_2(t)
 \leq \ln v + \dfrac{K^2}{\eps}\ln \dfrac1{\varepsilon} + (\lambda + 6\eps)t \;\;\forall t \in [T,\zeta_{u,v}],
\end{aligned}
\end{equation}
where we have used the facts $\varphi_u(t)\leq \varphi_H(t) \a.s\forall t\geq 0, u\in [H^{-1},H]$ and the non-decreasing property of $f(s,0)-\dfrac 12 g^2(s,0)$.
Therefore, for all $(u,v) \in [H^{-1},H]\times (0,\delta]$, $\omega\in \cap_{i=1}^{7}\Omega _i^{u,v}$ we get 
\begin{equation}\label{I}
\begin{aligned}
I_{u,v}(t)\leq  v\eps^{-\frac{K^2}{\eps}}e^{ \left(\lambda + 6\eps\right)t}\;\;\forall t \in [T,\zeta_{u,v}].
\end{aligned}
 \end{equation}
To proceed, we will estimate the difference $\abs{S_{u,v}(t)-\varphi_u(t)}$. 
By using  It\^o's formula and variation of constant formula, we get from \eqref{e2.1} and \eqref{phi} that
\begin{equation}\label{def1}
\begin{aligned}
S_{u,v}(t)-\varphi_u(t)&=e^{-c_1t+\sigma_1B_1(t)}\int_0^t e^{c_1s-\sigma_1B_1(s)}I_{u,v}(s)f(S_{u,v}(s),I_{u,v}(s))ds\\
&\quad+e^{-c_1t+\sigma_1B_1(t)}\int_0^t e^{c_1s-\sigma_1B_1(s)}I_{u,v}(s)g(S_{u,v}(s),I_{u,v}(s))dB_3(s)
\\&:=A_1^{u,v}(t)+A_2^{u,v}(t).
\end{aligned}
\end{equation}
In the next paragraph, we are going to estimate $A_1^{u,v}(t),A_2^{u,v}(t)$.

A consequence of Lemma \ref{phi} is that 
there is $\Omega_8^{u}$ satisfying $\PP(\Omega_8^u)\geq 1-\eps$, and $L_1$ (independent of $u$) such that for any $u\in [H^{-1},H]$, $\omega\in\Omega_8^u$, one has
$
\varphi_u(t)\leq L_1e^{\eps t}.
$
 By direct calculations,
from \eqref{I} and the assumption \ref{conditionfg} we obtain that $\forall (u,v)\in [H^{-1},H]\times(0,\delta], \omega \in\cap_{i=1}^{8}\Omega _i^{u,v},t\geq T$
\begin{equation*}
\begin{aligned}
\big|&A_1^{u,v}(t\wedge \zeta_{u,v})\big|\leq  e^{-c_1(t\wedge \zeta_{u,v})+\sigma_1B_1(t\wedge \zeta_{u,v})}\int_0^T e^{c_1s-\sigma_1B_1(s)}I_{u,v}(s)f(S_{u,v}(s),I_{u,v}(s))ds\\
&\quad+ e^{-c_1(t\wedge \zeta_{u,v})+\sigma_1B_1(t\wedge \zeta_{u,v})}\int_T^{t\wedge \zeta_{u,v}} e^{c_1s-\sigma_1B_1(s)}I_{u,v}(s)f(S_{u,v}(s),I_{u,v}(s))ds\\
&\leq2\eta F\bar H e^{-c_1(t\wedge \zeta_{u,v})+q_\eps (t\wedge \zeta_{u,v})}\int_0^T e^{c_1s+q_\eps(s)}ds
\\&\quad+Fe^{-c_1(t\wedge \zeta_{u,v})+q_\eps(t\wedge \zeta_{u,v})}\int_T^{t\wedge \zeta_{u,v}} e^{c_1s+q_\eps(s)}I_{u,v}(s)\big(S_{u,v}(s)+I_{u,v}(s)\big)ds\\
&\leq  2\eta F\bar H\Phi_1(\eps)+Fe^{-c_1(t\wedge \zeta_{u,v})+q_\eps(t\wedge \zeta_{u,v})}\int_T^{t\wedge \zeta_{u,v}} e^{c_1s+q_\eps(s)}I_{u,v}(s)\big(\nu+\varphi_u(s)+I_{u,v}(s)\big)ds\\
&\leq  2 \eta F\bar H\Phi_1(\eps)+vF{\eps}^{-\frac{K^2}{\eps}}e^{-c_1(t\wedge \zeta_{u,v})+q_\eps(t\wedge \zeta_{u,v})}\left[\int_T^{t\wedge \zeta_{u,v}} \hskip -5mm e^{c_1s+q_\eps(s)}e^{(\lambda+6\eps)s}(\nu+L_1e^{\eps s})ds\right.\\
&\hskip 9cm+\left.\int_T^{t\wedge \zeta_{u,v}} \hskip -5mm  e^{c_1s+q_\eps(s)}e^{2(\lambda+6\eps)s}ds\right].
\end{aligned}
\end{equation*}
In addition
\begin{equation*}
\begin{aligned}
&\sup_{t\geq 0}F{\eps}^{-\frac{K^2}{\eps}}e^{-c_1t+q_\eps(t)}\int_T^{t}  e^{c_1s+q_\eps(s)}e^{(\lambda+6\eps)s}(\nu+L_1e^{\eps s})ds:=\Phi_2(\eps)<\infty,\\
&\sup_{t\geq 0}F{\eps}^{-\frac{K^2}{\eps}}e^{-c_1t+q_\eps(t)}\int_T^{t}   e^{c_1s+q_\eps(s)}e^{2(\lambda+6\eps)s}ds:=\Phi_3(\eps)<\infty.
\end{aligned}
\end{equation*}
As a consequence
\begin{equation}\label{def2}
\abs{A_1^{u,v}(t\wedge \zeta_{u,v})}\leq  2\eta F\bar H\Phi_1(\eps)+v\big(\Phi_2(\eps)+\Phi_3(\eps)\big).
\end{equation}
Similarly, $\forall (u,v)\in [H^{-1},H]\times(0,\delta], \omega \in\cap_{i=1}^{8}\Omega _i^{u,v},t\geq T$ 
\begin{equation*}
\begin{aligned}
&\abs{A_2^{u,v}(t\wedge \zeta_{u,v})}\leq  \eta e^{-c_1(t\wedge \zeta_{u,v})+q_\eps(t\wedge \zeta_{u,v})}e^{c_1T+q_\eps(T)}\abs{\int_0^T g(S_{u,v}(s),I_{u,v}(s))dB_3(s)ds}
\\&\quad+ e^{-c_1(t\wedge \zeta_{u,v})+\sigma_1B_1(t\wedge \zeta_{u,v})}\abs{\int_T^{t\wedge \zeta_{u,v}} e^{c_1s-\sigma_1B_1(s)}I_{u,v}(s)g(S_{u,v}(s),I_{u,v}(s))dB_3(s)}\\
&\leq \eta \Phi_1(\eps)+e^{-c_1(t\wedge \zeta_{u,v})+q_\eps(t\wedge \zeta_{u,v})}\int_T^{t\wedge \zeta_{u,v}}q_\eps\sqrt{n(s)(\abs{\ln n(s)}+1)}ds,
\end{aligned}
\end{equation*}
where $n(t)=\int_T^t e^{2c_1s+2q_\eps (s)}I^2_{u,v}(s)g^2(S_{u,v}(s),I_{u,v}(s))ds$.
By \eqref{I} and boundedness of $g(s,i)$ we obtain that 
\begin{equation*}
\begin{aligned}
n(t)
\leq v^2K^2\eps^{-\frac{2K^2}{\eps}}\int_T^t e^{2c_1s+2q_\eps(s)}e^{2(\lambda+6\eps)s}ds.
\end{aligned}
\end{equation*}
Therefore, by a similar argument in the processing of getting \eqref{def2}, there exists $\Phi_{4}(\eps)$ such that for all $(u,v)\in [H^{-1},H]\times (0,\delta],$ $\omega\in \cap_{i=1}^{8}\Omega _i^{u,v}$, $t\geq T$
\begin{equation}\label{def3}
\abs{A_2^{u,v}(t\wedge \zeta_{u,v})}\leq \eta \Phi_1(\eps)+v\Phi_4(\eps).
\end{equation}
Let $\hat \delta\in (0,\delta)$ be a constant satisfying
$$\hat\delta\big[\Phi_2(\eps)+\Phi_3(\eps)+\Phi_4(\eps)\big]<\dfrac{\nu}2\quad\text{and}\quad \hat \delta\eps^{-\frac{K^2}{\eps}}e^{(\lambda + 6\eps)T}<\eta.$$
Hence, by combining \eqref{def1}, \eqref{def2} and \eqref{def3}, we obtain that for all $(u,v)\in [H^{-1}, H]\times (0,\hat\delta]$ and $\omega\in \cap_{i=1}^{8}\Omega _i^{u,v}$, $t\geq T$
\begin{equation*}
\begin{aligned}
\big|S_{u,v}(t\wedge \zeta_{u,v})-&\varphi_u(t\wedge\zeta_{u,v})\big|\\
&\leq \eta \big[2F\bar H\Phi_1(\eps)+\Phi_1(\eps)\big]+v\big[\Phi_2(\eps)+\Phi_3(\eps)+\Phi_4(\eps)\big]<\dfrac \nu 2+\dfrac \nu 2=\nu.
\end{aligned}
\end{equation*}
It follows that $t\wedge \zeta_{u,v} \leq \xi_{u,v}^\nu\;\forall t\geq T$. Therefore, for all $\omega\in \cap_{i=1}^{8}\Omega _i^{u,v}$, $\zeta_{u,v}\leq \xi^\nu_{u,v}$ and the equality only occurs when $\zeta_{u,v}=\xi_{u,v}^\nu=\infty$. As a consequence, $\cap_{i=1}^{8}\Omega _i^{u,v}\subset \{\tau_{u,v}^\eta\leq \xi_{u,v}^\nu\}$. Hence, combining with \eqref{I} we have $\forall (u,v)\in [H^{-1},H]\times(0,\hat\delta],\omega\in \cap_{i=1}^{8}\Omega _i^{u,v}, t\geq T$
\begin{equation*}
I_{u,v}(t\wedge \tau_{u,v}^{\eta})\leq \hat \delta\eps^{-\frac{K^2}{\eps}}e^{(\lambda + 6\eps)T}< \eta .
\end{equation*}
That means $t\wedge \tau_{u,v}^\eta<\tau_{u,v}^\eta\;\forall t\geq T$ or $\tau_{u,v}^\eta=\infty$ for all $(u,v)\in [H^{-1},H]\times(0,\hat\delta]$ and $\omega\in \cap_{i=1}^{8}\Omega _i^{u,v}$.
As a result, by the assumption \ref{conditionfg} we have $\forall (u,v)\in [H^{-1},H]\times(0,\hat\delta]$, $\omega\in \cap_{i=1}^{8}\Omega _i^{u,v}$
\begin{equation*}
\begin{aligned}
\limsup\limits_{t\to\infty}&\abs{\dfrac {\ln I_{u,v}(t)}t-\lambda}\leq\dfrac {\ln v}t + \limsup\limits_{t\to\infty}\dfrac 1t\int_0^t \Big[f(S_{u,v}(s),I_{u,v}(s))-f(\varphi_{u}(s),0)\Big]ds
\\&\;\;\;\; +\limsup\limits_{t\to\infty}\dfrac 1{2t}\int_0^t \Big[g^2(S_{u,v}(s),I_{u,v}(s))-g^2(\varphi_u(s),0)\Big]ds
\\&\;\;\;\;+\limsup\limits_{t\to\infty}\dfrac {\sigma_2B_2(t)}{t}+\limsup\limits_{t\to\infty}\dfrac 1t\int_0^t g(S_{u,v}(s),I_{u,v}(s))dB_3(s)
\\&\leq F(\nu+\eta)+KG(\nu+\eta)= \nu(F+KG)+\eta(F+KG)<\eps.
\end{aligned}
\end{equation*}
The proof is competed by noting that $\PP(\cap_{i=1}^{8}\Omega _i^{u,v})\geq 1-8\eps.$
In the case of $f(s,0)-\dfrac 12 g^2(s,0)$ is non-increasing, this proposition is similarly proved by choosing $\varphi_{H^{-1}}(t)$ in \eqref{lni2} instead of $\varphi_H(t)$.
\end{proof}

\begin{proof}[Proof of Theorem \ref{R<1}]
Let $0<\varepsilon<\min\{\frac 19,\frac{-\lambda}9\}$ and initial point $(u',v')\in\R_+^{2,\circ}$ be arbitrary. 
Choose
$$H\geq\max\Big\{1;\Big(\frac{3Q_1}{\varepsilon}\Big)^{\frac1{1+p}}; 2\Big(\frac{3Q_1}{\varepsilon}\Big )^{\frac 1{\bar p}}\Big\}.$$
We obtain from Lemma \ref{lem2.1} and Chebyshev's inequality that
\begin{equation*}
\limsup_{t\to\infty}\mathbb{P}\Big\{ S_{u',v'}(t)\geq H\Big\}=\limsup_{t\to\infty}\mathbb{P}\Big\{ S^{1+p}_{u',v'}(t) \geq H^{1+p} \Big\} \leq \limsup_{t\to\infty}\dfrac{\mathbb{E}S^{1+p}_{u',v'}(t)}{H^{1+p}}\leq \dfrac{\varepsilon}{3},
\end{equation*}
\begin{equation*}
\limsup_{t\to\infty}\mathbb{P}\Big\{ I_{u',v'}(t)\geq H\Big\}=\limsup_{t\to\infty}\mathbb{P}\Big\{ I^{1+p}_{u',v'}(t) \geq H^{1+p} \Big\} \leq \limsup_{t\to\infty}\dfrac{\mathbb{E}I^{1+p}_{u',v'}(t)}{H^{1+p}}\leq \dfrac{\varepsilon}{3},
\end{equation*}
and
\begin{equation*}
\begin{aligned}
\limsup_{t\to\infty}\mathbb{P}\Big\{ S_{u',v'}(t)+I_{u',v'}(t) \leq 2H^{-1} \Big\}&=\limsup_{t\to\infty}\mathbb{P}\Big\{ [S_{u',v'}(t)+I_{u',v'}(t)]^{-\bar p} \geq 2^{-\bar p}H^{\bar p} \Big\} 
\\& \leq \limsup_{t\to\infty}\dfrac{\mathbb{E}[S_{u',v'}(t)+I_{u',v'}(t)]^{-\bar p}}{2^{-\bar p}H^{\bar p}}
\leq \dfrac{\varepsilon}{3}.
\end{aligned}
\end{equation*}
Hence, it is seen that
\begin{equation}\label{2.1}
\limsup_{t\to\infty}\mathbb{P}\big\{(S_{u',v'}(t),I_{u',v'}(t)) \in \A \big\}\geq 1-\varepsilon,
\end{equation}
where $\A=\big\{(s, i) : 0\leq s \leq H,0<i\leq H, s+i\geq 2H^{-1}\big\}$.
By Proposition \ref{lm3.1}, there exists $ \hat\delta\in (0,H^{-1})$ such that
\begin{equation}\label{2.4}
\mathbb{P}\Bigg\{\limsup_{t\to\infty}\abs{\dfrac{\ln I_{u,v}(t)}{t}- \lambda}\leq \eps \Bigg\}\geq 1-8\varepsilon\;\forall\,(u,v) \in [H^{-1},H]\times(0, \hat\delta].
\end{equation}
That means the process $(S_{u',v'}(t),I_{u',v'}(t))$ is not recurrent (see e.g., \cite{WK} for definition) in the invariant set $\M=\{(s,i): 0\leq s, 0< i\}$. Because the diffusion equation \eqref{e2.1}  is non-degenerate, its solution process is either recurrent or transient (see e.g., \cite[Theorem 3.2]{WK}).
As a result, it must be transient.
Denote by $\A_1=\big\{(s, i) : 0\leq s\leq H,\hat\delta \leq i\leq H, s+i\geq 2H^{-1}\big\}$ a compact subset of $\M$. By transient property of $(S_{u',v'}(t),I_{u',v'}(t))$
\begin{equation}\label{2.2}
\lim_{t\to\infty}\mathbb{P}\big\{(S_{u',v'}(t),I_{u',v'}(t))\in \A_1\big\}=0.
\end{equation}
Combining \eqref{2.1}, \eqref{2.2} and $ \A \backslash \A_1 \subset [H^{-1},H]\times(0, \hat\delta] $ we have
\begin{equation*}
\limsup_{t\to\infty}\mathbb{P}\big\{(S_{u',v'}(t),I_{u',v'}(t)) \in [H^{-1},H]\times(0, \hat\delta]\big\}\geq 1-\varepsilon.
\end{equation*}
Therefore, there exists $T_3$ such that
\begin{equation}\label{2.3}
\mathbb{P}\big\{(S_{u',v'}(T_3),I_{u',v'}(T_3))\in [H^{-1},H]\times(0,\hat\delta]\big\}\geq 1-2\varepsilon.
\end{equation}
The Markov property, \eqref{2.4} and \eqref{2.3} deduce that
\begin{equation*}
\mathbb{P}\Bigg\{\limsup_{t\to\infty}\abs{\dfrac{\ln I_{u',v'}(t)}{t}- \lambda} \leq \eps\Bigg\}\geq(1-2\varepsilon)(1-8\varepsilon)\geq 1-10\varepsilon.
\end{equation*}
Since $\varepsilon$ is arbitrary, 
\begin{equation*}
\mathbb{P}\Big\{\limsup_{t\to\infty}\dfrac{\ln I_{u',v'}(t)}{t}= \lambda \Big\}=1.
\end{equation*}

We move to the proof of second part. Let $0>\bar \lambda > \max\{\lambda,-c_1\}$ be arbitrary and  $\bar\eps>0$ such that $\bar \lambda -\bar \epsilon > \max\{\lambda,-c_1\}$. 
We have
\begin{equation}\label{17}
e^{-\bar\lambda t}\abs{S_{u',v'}(t)-\varphi_{u'}(t)}\leq e^{-\bar\lambda t}\abs{A_1^{u',v'}(t)}+e^{-\bar\lambda t}\abs{A_2^{u',v'}(t)},
\end{equation}
where $A_1^{u',v'},A_2^{u',v'}$ are determined as in \eqref{def1}. 
We can obtain that 
there exists finite random variables $L_3=L_3(u',v')$, depending on $u',v'$, such that 
$S_{u',v'}(t)\leq L_3e^{\bar\eps t}\quad\forall t\geq 0\a.s$
Therefore,  the fact $\limsup_{t\to\infty}\frac{\ln I_{u',v'}(t)}{t}=\lambda, \lim_{t\to\infty}\frac {B_1(t)}t=0\a.s$ and the assumption \ref{conditionfg} imply that there exists a positive finite random variable $L_4=L_4(u',v')$ satisfying
\begin{equation*}
\begin{aligned}
e^{-c_1+\sigma_1B_1(t)}\int_0^t &e^{c_1s-\sigma_1B_1(s)}I_{u',v'}(s)f(S_{u',v'}(s),I_{u',v'}(s))ds
\\&\leq L_4 e^{-(c_1-\frac {\bar \eps}5) t}\int_0^t e^{(c_1+\frac {\bar \eps}5)s}e^{(\lambda+\frac{\bar\eps}5)s}(L_3e^{\bar\eps s}+e^{(\lambda+\frac{\bar\eps}5)s})ds.
\end{aligned}
\end{equation*}
Hence, using L'Hospital's rule yields (see \cite[proof of Theorem 2.2]{NHU1} for detailed calculations) yields
\begin{multline}\label{18}
\lim_{t\to\infty}e^{-\bar\lambda t}\abs{A_1^{u',v'}(t)}\\
\leq \lim_{t\to\infty}L_4e^{-(c_1+\bar\lambda-\frac {\bar \eps}5 )t}\int_0^t e^{(c_1+\frac {\bar \eps}5) s}e^{(\lambda+\frac{\bar\eps}5)s}(L_3e^{\bar\eps s}+e^{(\lambda+\frac{\bar\eps}5)s})ds
=0\a.s
\end{multline}
On the other hand, as in the proof of Lemma \ref{br}
$$\limsup_{t\to\infty}\dfrac{\abs{e^{-c_1t+\sigma_1B_1(t)}\int_0^t e^{c_1s-\sigma_1B_1(s)}I_{u,v}(s)g(S_{u,v}(s),I_{u,v}(s))dB_3(s)}}{e^{-c_1t+\sigma_1B_1(t)}\sqrt {n_1(t)(\abs{\ln n_1(t)}+1)}}<\infty\a.s,$$
where $n_1(t)=\int_0^t e^{2c_1s-2\sigma_1B_1(s)}I^2_{u,v}(s)g^2(S_{u,v}(s),I_{u,v}(s))ds$.
That means
\begin{equation}\label{19}
\limsup_{t\to\infty}\dfrac{e^{-\bar\lambda t}\abs{A_2^{u',v'}(t)}}{e^{-\bar\lambda t}e^{-c_1t+\sigma_1B_1(t)}\sqrt {n_1(t)(\abs{\ln n_1(t)}+1)}}<\infty\a.s
\end{equation}
By the facts $g(s,i)\leq K\;\forall s,i\geq 0$, $\limsup_{t\to\infty}\frac{\ln I_{u',v'}(t)}{t}=\lambda\a.s$ and $\limsup_{t\to\infty}\frac{\sigma_1B_1(t)}{t}=0\a.s$, there exists a positive finite random variable $L_5=L_5(u',v')$ such that
\begin{equation*}
e^{-\bar\lambda t}e^{-c_1t+\sigma_1B_1(t)}\sqrt {n_1(t)(\abs{\ln n_1(t)}+1)}\leq L_5e^{-(c_1+\bar\lambda-\frac{\bar\eps}5)t}\sqrt {n_2(t)(\abs{\ln n_2(t)}+1)}\a.s
\end{equation*}
where $n_2(t)=e^{(2c_1+2\lambda+\frac{\bar\eps}{5})t}$.
Therefore, it is easy to see that
\begin{equation}\label{20}
\limsup_{t\to\infty}e^{-\bar\lambda t}e^{-c_1t+\sigma_1B_1(t)}\sqrt {n_1(t)(\abs{\ln n_1(t)}+1)}=0\a.s
\end{equation}
Combining \eqref{19} and \eqref{20} we have
\begin{equation}\label{21}
\limsup_{t\to\infty}e^{-\bar\lambda t}\abs{A_2^{u',v'}(t)}=0\a.s
\end{equation}
Hence, \eqref{17},\eqref{18} and \eqref{21} imply that
$$\limsup_{t\to\infty}e^{-\bar\lambda t}\abs{S_{u',v'}(t)-\varphi_{u'}(t)}=0\a.s$$
This means
\begin{equation*}
\limsup_{t\to \infty}\frac{\ln \abs{S_{u',v'}(t)-\varphi_{u'}(t)}}t\leq \max\{\lambda,-c_1\}.
\end{equation*}
The proof is complete.
\end{proof}
\section {Permanence}\label{secper}
In this section, we deal with  the case $R>1$ (equivalently, $\lambda>0$). Because the proofs are rather technical, we  explain briefly the main ideas and steps to obtain the results before giving the detailed proofs.
\begin{thm}\label{R>1}
Assume that $\lambda>0$. We also assume that the function $f(s,0)-\dfrac 12 g^2(s,0)$ is non-decreasing. Then for any initial  point $(u,v)\in \mathbb{R}_+^{2,\circ}$, the system \eqref{e2.1} is permanent, i.e.,  the solution $(S_{u,v}(t),I_{u,v}(t))$ has a unique invariant probability $\pi^*(\cdot)$ concentrated on $\mathbb{R}_+^{2,\circ}$. Moreover,
\begin{itemize}
\item[{\rm (a)}] 
For any $(u,v)\in\mathbb{R}^{2,\circ}_+,$ 
\begin{equation*}\label{etv}
\lim\limits_{t\to\infty} t^{q^*}\|P(t, (u,v), \cdot)-\pi^*(\cdot)\|=0,
\end{equation*}
where $\|\cdot\|$ is the total variation norm, $q^*$ is any positive number and $P(t, (u, v), \cdot)$ is the transition probability of $(S_{u,v}(t), I_{u,v}(t))$.
\item[{\rm(b)}] The strong large law number holds, i.e, for any $\pi^*$-integrable $ h: \mathbb{R}^{2,\circ}_+\to\mathbb{R}$, we have
$$\lim\limits_{t\to\infty}\frac1t\int_0^th(S_{u,v}(s), I_{u,v}(s))ds=\int_{\mathbb{R}^{2,\circ}_+} h(x,y)\pi^*(dx, dy) \;a.s.\,\forall \,(u,v)\in\mathbb{R}^{2,\circ}_+.$$
\end{itemize}
\end{thm}
The main idea to prove this theorem is similar to one in \cite {DANG}. That is to construct a function $V_*:\R_+^{2,\circ}\rightarrow [1,\infty)$ satisfying
\begin{equation}\label{e4.2}\E V_*(S_{u,v}(t^*),I_{u,v}(t^*))\leq V_*(u,v)-P_1^*V_*^{\gamma}(u,v)+P_2^*\1_{\{u,v\in \K\}} \;\;\;\;\;\forall (u,v)\in\R_+^{2,\circ},
\end{equation}
for some petite set $\K$ and some $\gamma\in (0,1)$, $P_1^*,P_2^*>0,t^*>1$ and then apply \cite[Theorem 3.6]{JR}. We also refer the reader to \cite[ pp.106 -124]{MT} for further details on petite sets. Basing on the definition of the value  $\lambda$ in Section \ref{secext}, we will construct $V_*$ as a sum of the Lyapunov function $V_1(u,v)$ defined  in the Lemma \ref{lem2.1} and the function $\ln^- v:=\max\{-\ln v,0\}.$ 
 If the solution starts from a initial point $(u,v)$ with  sufficient small  $v$, the functions $\ln^- I_{u,v}(t^*),\ln^- v$ are utilized to dominate the inequality \eqref{e4.2} (see Propositions \ref{prop2.3} and \ref{prop2.4}) while in the remaining region, the Lyapunov functions $V_1(S_{u,v}(t^*),I_{u,v}(t^*)), V_1(u,v)$ play an important role (by using Lemma \ref{lem2.1}). 
Lemmas \ref {lem2.4}, \ref{lem2.5} and \ref{lem2.6} are auxiliary results needed for Propositions \ref{prop2.3} and \ref{prop2.4}.
\begin{lem}\label{lem2.4}
There are positive constants $K_1, K_2$ such that, for any $t\geq 1$ and $A\in\F$
\begin{equation*}
\E\big(( \ln^- I_{u,v}(t))^{2}\1_A\big)\leq \PP(A)(\ln^- v)^2+ K_1 \sqrt{\PP(A)}\ln^- v\;t+K_2\sqrt{{\PP(A)}}\;t^2,
\; \forall (u,v) \in \R_{+}^{2}, v\neq 0.
\end{equation*}
\end{lem}
\begin{proof}
For any initial point $(u,v)\in \R_+^{2}$, $v\neq 0$, we obtain from \eqref{lni} that
  \begin{align*}
-\ln I_{u,v}(t)
\leq  -\ln v+ \left(c_2+\dfrac{K^2}2\right)t+\sigma_2\abs{B_2(t)}+\abs{\int_0^tg(S_{u,v}(s),I_{u,v}(s))dB_3(s)}.
\end{align*}
Hence
\begin{equation*}
\ln^- I_{u,v}(t)\leq \ln^- v+\left(c_2+\dfrac{K^2}2 \right)t+\sigma_2\abs{B_2(t)}+\abs{\int_0^tg(S_{u,v}(s),I_{u,v}(s))dB_3(s)}.
\end{equation*}
From the inequality $(x_1+x_2+x_3+x_4)^2\leq x_1^2+3(x_2^2+x_3^2+x_4^2)+2x_1(x_2+x_3+x_4)$ we get
\begin{equation*}
\begin{aligned}
\hskip -2mm (\ln^- &I_{u,v})^2\1_A\leq \hskip -1mm (\ln^- v)^2\1_A+3t^2\Big(c_2+\hskip -1mm \frac{K^2}2\Big)^2\1_A+3 \sigma_2^2B_2^2(t)\1_A+3\abs{\int_0^t\hskip-2mm g(S_{u,v}(s),I_{u,v}(s))dB_3(s)}^2\hskip-2mm \1_A\\
&+2\ln^- v\Big(c_2+\frac{K^2}2\Big)t\1_A+2\sigma_2\ln^- v\abs{B_2(t)}\1_A+2\ln^- v\abs{\int_0^tg(S_{u,v}(s),I_{u,v}(s))dB_3(s)}\1_A.
\end{aligned}
\end{equation*}
Applying H\"older's inequality and Burkholder-Davis-Gundy inequality, we obtain 
\begin{align*}
&\E \abs{B_2(t)}\1_A\leq \sqrt{\PP(A)}\sqrt{\E B_2^2(t)}\leq \sqrt{\PP(A)}\sqrt t;\;\; \E {B_2^2(t)}\1_A\leq \sqrt{\PP(A)}\sqrt{\E B_2^4(t)}\leq 3\sqrt{\PP(A)}\,t;\\
&\E \abs{\int_0^t g(S_{u,v}(s),I_{u,v}(s))dB_3(s)}\1_A\leq  \sqrt{\PP(A)}\left(\int_0^t g^2(S_{u,v}(s),I_{u,v}(s))ds\right)^{\frac12}\leq K\sqrt{\PP(A)}\sqrt t,
\end{align*}
and
\begin{align*}
&\E \abs{\int_0^t g(S_{u,v}(s),I_{u,v}(s))dB_3(s)}^2\1_A\leq \sqrt{\PP(A)} \left(\E\abs{\int_0^t g(S_{u,v}(s),I_{u,v}(s))dB_3(s)}^{4}\right)^{\frac12}\\
&\quad\quad\leq \sqrt{\PP(A)} \left(3\E\abs{\int_0^t g^2(S_{u,v}(s),I_{u,v}(s))ds}^{2}\right)^{\frac12}\leq 3\sqrt{\PP(A)}K^2t.
\end{align*}
Therefore, there exist two constants $K_1, K_2$ such that 
\begin{equation*}
\E(\ln^- I_{u,v}(t))^{2}\1_A\leq \PP(A)[\ln^- v]^2+ K_1\sqrt{\PP(A)}\ln^- v\, t+K_2\sqrt{{\PP(A)}}\,t^2.
\end{equation*}
Lemma is proved.
\end{proof}
\begin{lem}\label{lem2.5}
For any $\eps>0$, there is a constant $M(\eps)>0$ such that
$$\PP\left\{\abs{\sigma_2B_2(t)+\int_0^t g(S_{u,v}(s),I_{u,v}(s))dB_3(s)}\leq M(\eps){t^{\frac23}} ,\;\forall t\geq 1\right\}\geq 1-\eps.$$
\end{lem}
\begin{proof} The proof follows from Lemma \ref{br} in paying attention that $g$ is bounded. 
\end{proof}
Choose an integer number $m>2$ such that $\dfrac m{m-1}<\dfrac {\lambda}{12c_2}+1$ and a sufficiently small number $\eps^*\in(0,1)$ satisfying
\begin{equation}\label{e4.4}
  \frac{3\lambda}2(1-\eps^*)-K_1\sqrt{\eps^*}> \lambda\;\;
\text{ and }\;\;
\dfrac 34\lambda(1-5\eps^*)-mK_1\sqrt{5\eps^*}>\frac{\lambda}2.
\end{equation}
By the non-decreasing property of $f(s,0)-\frac 12 g^2(s,0)$ and the definition of $\lambda$, there exists $H^*>0$ such that $\inf\limits_{s\geq H^*} \Big\{f(s,0)-\dfrac 12 g^2(s,0)\Big\}\geq c_2+\dfrac {3\lambda}4.$ 
In what follows, to simplify notations, we suppress the superscript $(u,v)$ on $(\Omega_i^*)^{u,v}, (\Omega_i^{**})^{u,v}$  if there is no confusion  although they  may depend on $(u,v)$. Moreover, $m$, $\eps^*$ and $H^*$ satisfying the above conditions are fixed.
\begin{lem}\label{lem2.6}
  For $m, \eps^*$, $H^*$ chosen as above, there are $\delta_1^*\in(0,1)$ and $T^*>1$ such that
  \begin{equation*}
    \PP\Big\{\ln v+\frac{3\lambda t}{4}\leq \ln I_{u,v}(t)<0\;\;\forall\, t\in [T^*, mT^*]\Big\}\geq 1-\eps^*,
  \end{equation*}
  for all $(u,v)\in [0,H^*]\times(0,\delta_1^*]$.
\end{lem}
\begin{proof} It is similar to the  proof of Proposition \ref{lm3.1}, we deduce from 
the ergodicity of $\varphi_0$ that
 there exists $T_1^*$, such that $\mathbb{P}(\Omega_1^*)\geq 1-\dfrac{\eps^*}{4}$, where
$$ \Omega_1^*=\left\{-c_2+\dfrac1t \int_0^t f(\varphi_{0}(s),0)ds-\dfrac1{2t} \int_0^tg^2(\varphi_{0}(s),0)ds  \geq \dfrac{26\lambda}{28},\;\forall t\geq T_1^*\right\}.$$
By Lemma \ref{lem2.5}, there exists $M=M(\eps^*)>0$ such that $\PP(\Omega_2^*)\geq 1-\dfrac{\eps^*}{4}$, where 
$$\Omega_2^*=\left\{\abs{\sigma_2B_2(t)+\int_0^t g(S_{u,v}(s),I_{u,v}(s))dB_3(s)}\leq Mt^{\frac23} ,\;\forall t\geq 1\right\}.$$
Let $T^*>\max\left\{1, T_1^*,\frac{28^3M^3}{\lambda^3},\frac{12^3M^3m^2}{\lambda^3},\frac{M^3m^2}{c_2^3}\right\}$  be a constant satisfying
$$\exp\Big\{-\frac{\min\{c_2,\frac{2\lambda}{3}\}\lambda T^*}{8(\sigma_2^2+K^2)}\Big\}\leq \eps^*.$$
By Lipschitz continuity, there exists $\nu^*>0$ such that if $ |s_1-s_2|\leq\nu^*$ and $\abs{i_1-i_2}<\nu^*$ then $\abs{f(s_1,i_1)-f(s_2,i_2)}\leq\dfrac{\lambda}{28}$ and $\abs{ g^2(s_1,i_1)- g^2(s_2,i_2)}\leq\dfrac{\lambda}{14}$. 
By Proposition \ref{extinction}, we can show that for $\nu^*$ chosen as above there exists $0<\theta<\min\{1, \nu^*\}$ such that $\forall (u,v)\in [0,H^*]\times(0, \theta],$ $\PP(\Omega_3^*)\geq 1-\dfrac{\eps^*}{4}$, where
$$\Omega_3^*=\big\{ \abs{S_{u,v}(t)-\varphi_u(t)}\leq \nu^* \;\forall 0\leq t \leq mT^*\wedge \tau_{u,v}^{\theta}\big\}\;\text{and}\;\tau_{u,v}^{\theta}=\inf\{t\geq 0: I_{u,v}(t)>\theta\}.$$
By Proposition \ref{lem2.2}, it can be shown that there exists $\delta_1^*\in (0,\theta)$ so that $\forall (u,v)\in [0,H^*]\times(0, \delta_1^*],$ $\PP(\Omega_4^*)\geq 1-\dfrac{\eps^*}{4}$, where
$$\Omega_4^*=\left\{\tau_{u,v}^{\theta}\geq mT^*\right\}.$$
From the equation \eqref{e2.1} and Ito's formula, by using a similar arguments in processing of getting \eqref {lni2} we obtain that $\forall (u,v)\in [0,H^*]\times(0, \delta_1^*],\;\omega \in \cap_{i=1}^4\Omega_i^*$ and  $T^*\leq t \leq mT^*$
\begin{equation*}
\begin{aligned}
0>\ln \theta\geq \ln I_{u,v}(t) \geq \ln v + \dfrac{22\lambda}{28}t-Mt^{\frac23} \geq \ln v + \dfrac{3\lambda}{4}t.
\end{aligned}
\end{equation*}
The proof is completed.
\end{proof}
\begin{prop}\label{prop2.3} Assume the assumptions in Theorem \ref{R>1} hold and let $T^*$ be  as in Lemma \ref{lem2.6}. Then, there exists a constant $K_3$ such that
  \begin{equation*}
    \E[\ln^- I_{u,v}(t)]^2\leq (\ln^- v)^2-\lambda t\ln^- v+K_3 t^2
  \end{equation*}
  for any $(u,v)\in [0,H^*]\times (0, \infty),\; t\in[T^*, mT^*].$
\end{prop}

\begin{proof}
  First, consider $v\in (0,\delta_1^*]$ where $\delta_1^*$ as in Lemma \ref{lem2.6} and $0\leq u\leq H^*.$  By Lemma \ref{lem2.6}, we obtain $\PP(\bar\Omega_{u,v})\geq1-\eps^*$ where
  $$\bar \Omega_{u,v}=\big\{\ln v+\dfrac{3\lambda t}{4}\leq \ln I_{u,v}(t)<0\;\;\;\forall\, t\in[T^*, mT^*]\big\}.$$
  Hence, in $\bar \Omega_{u,v}$ we have
  $$0\leq \ln^- I_{u,v}(t)\leq \ln^- v-\frac{3\lambda t}{4}\;\;\forall\, t\in[T^*, mT^*].$$
  As a result,
  $$[\ln^- I_{u,v}(t)]^2\leq (\ln^- v)^2-\frac{3\lambda t}2\ln^- v+\frac{9\lambda^2 t^2}{4}\;\;\forall\, t\in[T^*, mT^*],$$
  which implies that
  \begin{equation}\label{e4.8}
    \E\left[(\ln^- I_{u,v}(t))^2\1_{\bar\Omega_{u,v}}\right]\leq\PP(\bar\Omega_{u,v})(\ln^- v)^2-\frac{3\lambda t}{2}\PP(\bar\Omega_{u,v})\ln^- v+\frac{9\lambda^2 t^2}{4}\PP(\bar\Omega_{u,v}).
  \end{equation}
 In $\bar\Omega_{u,v}^c=\Omega\backslash \bar\Omega_{u,v},$ it follows from Lemma \ref{lem2.4} that
 \begin{equation}\label{e4.9}
    \E\left[\ln^- I_{u,v}(t)]^2\1_{\bar\Omega^c_{u,v}}\right]\leq\PP(\bar\Omega^c_{u,v})(\ln^- v)^2+K_1t \sqrt{\PP(\bar\Omega^c_{u,v})}\ln^- v+K_2t^2 \sqrt{\PP(\bar\Omega^c_{u,v})}.
 \end{equation}
 Adding \eqref{e4.8} and \eqref{e4.9} side by side obtains
  \begin{equation*}
    \E[\ln^- I_{u,v}(t)]^2\leq(\ln^- v)^2-\Big(\frac{3\lambda}2(1-\eps^*)-K_1\sqrt{\eps^*}\Big)t\ln^- v+\Big(\frac{9\lambda^2}{4}+K_2\Big)t^2 .
 \end{equation*}
 In view of \eqref{e4.4} we deduce that
  \begin{equation*}
    \E[\ln^- I_{u,v}(t)]^2\leq(\ln^- v)^2-\lambda t\ln^- v+\Big(\frac{9\lambda^2}{4}+K_2\Big)t^2 .
 \end{equation*}
 Now, for $v\in [\delta_1^*, \infty)$ and $0\leq u\leq H^*,$ it follows from Lemma \ref{lem2.4} that
 \begin{align*}
    \E[\ln^- I_{u,v}(t)]^2&\leq(\ln^- v)^2+K_1 t\ln^- v+K_2t^2
   \\&\leq |\ln\delta_1^*|^2+K_1 t|\ln\delta_1^*|+K_2t^2.
 \end{align*}
 Letting $K_3$ sufficiently large such that $K_3>\frac{9\lambda^2}{4}+K_2$ and
 $|\ln\delta_1^*|^2+(K_1+\lambda) t|\ln\delta_1^*|+K_2t^2\leq K_3t^2\;\;\forall t\in [T^*, mT^*]$,
 we obtain the desired result.
\end{proof}
\begin{prop}\label{prop2.4} Assume the assumptions in Theorem \ref{R>1} holds and let $T^*$ as in Lemma \ref{lem2.6}. There exists $K_5>0$  such that
  $$\E[\ln^- I_{u,v}(mT^*)]^2\leq (\ln^- v)^2-\frac{\lambda T^*}2\ln^- v+K_5 {T^*}^2,$$
  for any $v\in(0,\infty), u>H^*.$
\end{prop}
\begin{proof}
Let $\theta^*\in (0,1)$ be a constant that satisfies $\theta^*<\dfrac{\lambda}{12(F+KG)}$. 
By a similar argument to the proof of Proposition \ref{lem2.2}, there exists $\delta_1>0$ such that $\forall (u,v) \in (H^*,\infty)\times (0,\delta_1], A\in \F$
\begin{align*}
&\PP\bigg(\Omega_5^*\cap \left\{\int_0^{mT^*}\Big(f(S_{u,v}(s),I_{u,v}(s))-\dfrac 12g^2(S_{u,v}(s),I_{u,v}(s))\Big)ds\leq \Big(2c_2+\dfrac{(m-1)\lambda}2\Big)T^*\right\}\cap A\bigg)
\\&\geq \PP\bigg(\left\{\int_0^{mT^*}\Big(f(S_{u,v}(s),I_{u,v}(s))-\dfrac 12g^2(S_{u,v}(s),I_{u,v}(s))\Big)ds\leq \Big(2c_2+\dfrac{(m-1)\lambda}2\Big)T^*\right\}\cap A\bigg)-\eps^*,
\end{align*}
 where
$$\Omega_5^*=\{ \tau_{u,v}^{\theta^*} \geq mT^*\}\;\text{and}\;\tau_{u,v}^{\theta^*}=\inf\{t\geq 0: I_{u,v}(t)>\theta^*\}.$$
Now, consider $v\leq\delta_2^{*}:=\min \Big\{1,\delta_1,\theta^*\exp\big\{-1-(3c_2+M)mT^*-\frac{(m-1)\lambda T^*}{2}\big\}\Big\}$, $u>H^*$. Define the stopping time $\xi^*_{u,v}:=(m-1)T^*\wedge\inf\{t>0: S_{u,v}(t)\leq H^*\}$
and the following sets
\begin{equation*}
\begin{aligned}
&\Omega_6^*=\Omega_2^*\cap \{\xi_{u,v}^*=(m-1) T^*\}
\\&\bigcap \left\{\int_0^{mT^*} \Big(f(S_{u,v}(s),I_{u,v}(s))-\dfrac 12g^2(S_{u,v}(s),I_{u,v}(s))\Big)ds\leq (2c_2+\dfrac{(m-1)\lambda}2)T^*\right\}\cap \Omega_5^*,\\
&\Omega_7^*=\Omega_2^*\cap \{\xi_{u,v}^*= (m-1)T^*\}
\\&\bigcap \left\{\int_0^{mT^*} \Big(f(S_{u,v}(s),I_{u,v}(s)-\dfrac 12g^2(S_{u,v}(s),I_{u,v}(s))\Big)ds> (2c_2+\dfrac{(m-1)\lambda}2)T^*\right\},
\\
&\Omega_8^*=\{-\ln I_{u,v}(\xi_{u,v}^*)\leq -\ln v+\frac{\lambda T^*}8\}\cap\{\xi_{u,v}^*<(m-1)T^*\}; 
\Omega_{9}^*=\Omega\backslash(\Omega_6^*\cup\Omega_7^*\cup \Omega_8^*),
\end{aligned}
\end{equation*}
where $\Omega_2^*$ as in Lemma \ref{lem2.6}.
Our idea in this Proposition is to estimate $[\ln^- I_{u,v}(mT^*)]^2$ in each set $\Omega_6^*,\Omega_7^*,\Omega_8^*,\Omega_9^*$ by using Lemma \ref{lem2.4}.
First, for all $\omega\in \Omega_6^*,$ we have
\begin{equation*}
\begin{aligned}
&-\ln I_{u,v}(mT^*)=-\ln v-\int_0^{mT^*}\big[-c_2+f(S_{u,v}(s),I_{u,v}(s))-\dfrac 12 g^2(S_{u,v}(s),I_{u,v}(s))\big]ds
\\&\quad-\sigma_2 B_2(mT^*)-\int_0^{mT^*}g(S_{u,v}(s),I_{u,v}(s))dB_3(s)
  \\&\leq-\ln v+mc_2T^*
\\&\;\;\;\;\;-\int_0^{(m-1)T^*}\Big[f(S_{u,v}(s),I_{u,v}(s))-\dfrac 12 g^2(S_{u,v}(s),I_{u,v}(s))-f(S_{u,v}(s),0)+\dfrac 12 g^2(S_{u,v}(s),0)\Big]ds
\\&\;\;\;\;-\int_0^{(m-1)T^*}\Big(f(S_{u,v}(s),0)-\dfrac 12 g^2(S_{u,v}(s),0)\Big)ds+M(mT^*)^{\frac 23}
  \\&\leq-\ln v+mc_2T^*+(F+KG)\theta^* (m-1)T^*-(m-1)T^*(c_2+\dfrac {3\lambda}{4})+M(mT^*)^{\frac 23}
\\&\leq-\ln v-(m-1)T^*\Big(c_2+\dfrac {3\lambda}4-\dfrac m{m-1}c_2-\dfrac{Mm^{\frac 23}}{(T^*)^{\frac 13}}-(F+KG)\theta^*\Big)
\\&\leq -\ln v-\frac{(m-1)\lambda T^*}2\;\;\;\text{by the choice of}\;m,T^*,\theta^*.
\end{aligned}
\end{equation*}
Note that if $v\leq \delta_2^{*}$ then $-\ln v-\frac{(m-1)\lambda T^*}{2}>0$. Therefore
$$[\ln^- I_{u,v}(mT^*)]\leq -\frac{(m-1)\lambda T^*}{2}+\ln^- v.$$
By squaring and then multiplying by $\1_{\Omega_6^*}$ and taking expectation both sides, we yield
\begin{equation}\label{e4.11}
  \E\left( [\ln^- I_{u,v}(mT^*)]^2\1_{\Omega_6^*}\right)\leq (\ln^- v)^2\PP(\Omega_6^*)-(m-1)\lambda T^*\ln^- v\PP(\Omega_6^*)+\frac{(m-1)^2\lambda^2 {T^*}^2}{4}.
\end{equation}
Secondly, for all $\omega \in \Omega_7^*$, we also have
\begin{align*}
  -\ln &I_{u,v}(mT^*)=-\ln v-\int_0^{mT^*}\big[-c_2+f(S_{u,v}(s),I_{u,v}(s))-\dfrac 12 g^2(S_{u,v}(s),I_{u,v}(s))\big]ds
\\&\;\;\;\;\;-\sigma_2 B_2(mT^*)-\int_0^{mT^*}g(S_{u,v}(s),I_{u,v}(s))dB_3(s)
  \\&\leq-\ln v-\big(c_2+\dfrac{(m-1)\lambda}2\big)T^*+M(mT^*)^{\frac 23}\leq  -\ln v-\frac{(m-1)\lambda T^*}2.
\end{align*}
Therefore
\begin{equation}\label{e4.12}
  \E\left( [\ln^- I_{u,v}(mT^*)]^2\1_{\Omega_7^*}\right)\leq (\ln^- v)^2\PP(\Omega_7^*)-(m-1)\lambda T^*\ln^- v\PP(\Omega_7^*)+\frac{(m-1)^2\lambda^2 {T^*}^2}{4}.
\end{equation}
Thirdly, we will estimate $\PP(\Omega_8^*)$. Define the following sets, which help us in estimating $\PP(\Omega_8^*)$
\begin{align*}\Omega_1^{**}&=\{ \tau_{u,v}^{\theta^*} \geq \xi_{u,v}^*\},\\
\Omega_2^{**}&=\left\{-\sigma_2B_2(t)-\int_0^t g(S_{u,v}(s),I_{u,v}(s))ds \leq \min \{c_2,\dfrac {2\lambda}3\}t + \dfrac {\lambda T^*}8\;\;\forall t\geq 0\right\},\\
\Omega_3^{**}&=\Big\{\int_0^{\xi_{u,v}^*}\Big(f(S_{u,v}(s),I_{u,v}(s))-\dfrac 12g^2(S_{u,v}(s),I_{u,v}(s))\Big)ds\leq 2c_2mT^*\Big\}\cap \Omega_1^{**}\cap \Omega_2^{**},
\\
\Omega_4^{**}&=\Big\{\int_0^{\xi_{u,v}^*} \Big(f(S_{u,v}(s),I_{u,v}(s))-\dfrac 12g^2(S_{u,v}(s),I_{u,v}(s))\Big)ds> 2c_2mT^*\Big\}\cap \Omega_2^{**}.
\end{align*}
For all $\omega \in \Omega_3^{**}$, we have
\begin{equation*}
\begin{aligned}
  -&\ln I_{u,v}(\xi_{u,v}^*)\leq-\ln v+c_2\xi_{u,v}^*
\\&\;-\int_0^{\xi_{u,v}^*}\Big(f(S_{u,v}(s),I_{u,v}(s))-\dfrac 12 g^2(S_{u,v}(s),I_{u,v}(s))-f(S_{u,v}(s),0)+\dfrac 12 g^2(S_{u,v}(s),0)\Big)ds
\\&\;-\int_0^{\xi_{u,v}^*}\Big(f(S_{u,v}(s),0)-\dfrac 12 g^2(S_{u,v}(s),0)\Big)ds-\sigma_2 B_2(\xi_{u,v}^*)-\int_0^{\xi_{u,v}^*}g(S_{u,v}(s),I_{u,v}(s))dB_3(s)
\\&\leq-\ln v-\xi_{u,v}^*\Big(\dfrac {3\lambda}4-(F+KG)\theta^*\Big)-\sigma_2 B_2(\xi_{u,v}^*)-\int_0^{\xi_{u,v}^*}g(S_{u,v}(s),I_{u,v}(s))dB_3(s)
\\&\leq-\ln v-\xi_{u,v}^*\dfrac {2\lambda}3+\min\big\{c_2,\frac {2\lambda}3\big\}\xi_{u,v}^*+\dfrac{\lambda T^*}{8}
\leq -\ln v+\frac{\lambda T^*}8.
\end{aligned}
\end{equation*}
On the other hand, for all $\omega \in \Omega_4^{**}$ we have
\begin{equation*}
\begin{aligned}
  -\ln I_{u,v}&(\xi_{u,v}^*)=-\ln v-\int_0^{\xi_{u,v}^*}\big[-c_2+f(S_{u,v}(s),I_{u,v}(s))-\dfrac 12 g^2(S_{u,v}(s),I_{u,v}(s))\big]ds
\\&\quad-\sigma_2 B_2(\xi_{u,v}^*)-\int_0^{\xi^*_{u,v}}g(S_{u,v}(s),I_{u,v}(s))dB_3(s)
 \\& \leq-\ln v-c_2\xi_{u,v}^*+\min\big\{c_2,\frac {2\lambda}3\big\}\xi_{u,v}^*+\dfrac{\lambda T^*}{8}
\leq -\ln v+\frac{\lambda T^*}8.
\end{aligned}
\end{equation*}
As a consequence, 
\begin{equation}\label{e4.13}
\Big(\big(\Omega_3^{**}\cup \Omega_4^{**}\big)\cap\{\xi_{u,v}^*<(m-1)T^*\} \Big)\subset \Omega_8^*.
\end{equation}
By a similar way in processing to prove  Proposition \ref{lem2.2}, we obtain that $\forall (u,v)\in (H^*,\infty)\times(0,\delta_2^*]$
\begin{equation}
\begin{aligned}\label{e4.14}
\PP&\Big(\Omega_1^{**}\cap \Big\{\int_0^{\xi_{u,v}^*}\Big(f(S_{u,v}(s),I_{u,v}(s))-\dfrac 12g^2(S_{u,v}(s),I_{u,v}(s))\Big)ds\leq 2c_2mT^*\Big\}\Big)
\\&\geq \PP\Big\{\int_0^{\xi_{u,v}^*}\Big(f(S_{u,v}(s),I_{u,v}(s))-\dfrac 12g^2(S_{u,v}(s),I_{u,v}(s))\Big)ds\leq 2c_2mT^*\Big\}-\eps^*.
\end{aligned}
\end{equation}
Moreover, it is obvious that 
$$\dfrac {\min\{c_2,\frac {2\lambda}3\}}{2(\sigma_2^2+K^2)} \Big(\sigma_2^2t+\int_0^t g^2(S_{u,v}(s),I_{u,v}(s))ds\Big)+\frac{\lambda T^*}8<\min \Big\{c_2,\frac {2\lambda}3\Big\}t + \frac {\lambda T^*}8.$$ 
Therefore, the exponential martingale inequality \cite[Theorem 7.4, p. 44]{MAO} implies that
\begin{equation*}
\begin{aligned}
\PP&\Bigg\{ -\sigma_2B_2(t)-\int_0^t g(S_{u,v}(s),I_{u,v}(s))dB_3(s)\leq\min \Big\{c_2,\frac {2\lambda}3\Big\}t + \frac {\lambda T^*}8\;\forall t\geq 0\Bigg\}
\\&\quad\geq 1-\exp \Big\{-\frac{\lambda T^* \min\{ c_2,\frac {2\lambda}{3}\}}{8(\sigma_2^2+K^2)}\Big\}\geq 1-\eps^*.
\end{aligned}
\end{equation*}
That means $\PP(\Omega_2^{**})\geq 1-\eps^*.$ Therefore, we obtain from definition of $\Omega_3^{**}$ and \eqref{e4.14} that
\begin{equation}\label{e4.15}
\begin{aligned}
\PP&(\Omega_3^{**})\geq \PP\left(\left\{\int_0^{\xi_{u,v}^*}\Big(f(S_{u,v}(s),I_{u,v}(s))-\dfrac 12g^2(S_{u,v}(s),I_{u,v}(s))\Big)ds\leq \Big(2c_2+\dfrac{(m-1)\lambda}2\Big)T^*\right\}\cap \Omega_1^{**}\right)
\\&\quad\quad\quad\quad+\PP(\Omega_2^{**})-1
\\&\geq \PP\left\{\int_0^{\xi_{u,v}^*}\Big(f(S_{u,v}(s),I_{u,v}(s))-\dfrac 12g^2(S_{u,v}(s),I_{u,v}(s))\Big)ds\leq \Big(2c_2+\dfrac{(m-1)\lambda}2\Big)T^*\right\}-2\eps^*.
\end{aligned}
\end{equation}
On the other hand, by definition of $\Omega_4^{**}$ and the property of $\Omega_2^{**}$, we get
\begin{equation}\label{e4.16}
\begin{aligned}
\PP(\Omega_4^{**})\geq \PP\left\{\int_0^{\xi_{u,v}^*}\Big(f(S_{u,v}(s),I_{u,v}(s))-\dfrac 12g^2(S_{u,v}(s),I_{u,v}(s))\Big)ds> \Big(2c_2+\dfrac{(m-1)\lambda}2\Big)T^*\right\}-\eps^*.
\end{aligned}
\end{equation}
Thus, by the disjointedness of $\Omega_3^{**}$ and $\Omega_4^{**}$, we obtain from \eqref{e4.15} and \eqref{e4.16} that
\begin{equation}\label{e4.17}
\PP\Big(\Omega_3^{**}\cup \Omega_4^{**}\Big)\geq 1-3\eps^*.
\end{equation}
As a consequence of \eqref{e4.13} and \eqref{e4.17}, one has
$$\PP(\Omega_8^*)\geq \PP\{\xi_{u,v}^*<(m-1)T^*\}-3\eps^*.$$
In addition, by definition of $\Omega_6^{*},\Omega_7^{*}$, properties of $\Omega_5^*$, $\Omega_2^*$ and some basis computations, we obtain
$$\PP\big(\Omega_6^*\cup \Omega_7^*\big)\geq \PP\big(\Omega_2^*\cap \{\xi_{u,v}^*=(m-1) T^*\}\big)-\eps^*\geq \PP\big(\{\xi_{u,v}^*=(m-1) T^*\}\big)-2\eps^*.$$
Therefore, we have
$$\PP(\Omega_6^*\cup\Omega_7^* \cup\Omega_8^*)\geq 1-5\eps^*,$$
or $\PP(\Omega_9^*)\leq 5\eps^*$.
Let $t<(m-1)T^*$, $u'>0$ and $v'$ satisfy $-\ln v'\leq -\ln v+\frac{\lambda T^*}8\leq0$. By applying Proposition \ref{prop2.3} and the strong Markov property, we can estimate the following conditional expectation
\begin{align*}
\E\Big[[\ln^- I_{u,v}&(mT^*)]^2\Big{|}\xi_{u,v}=t, \;I_{u,v}(\xi_{u,v})=v', S_{u,v}(\xi_{u,v})=u'\Big]
\\&
\leq (\ln^- v')^2-\lambda(mT^*-t)\ln^- v'+K_3(mT^*-t)^2
\\&
\leq (\ln^- v')^2-\lambda T^*\ln^- v'+m^2K_3{T^*}^2
\\&
\leq \big(-\ln v+\frac{\lambda T^*}8\big)^2-\lambda T^*(-\ln v)+m^2K_3{T^*}^2
\\&
\leq \big(-\ln v\big)^2-\big(\lambda T^*-\frac{\lambda T^*}4\big)(-\ln v)+m^2K_3{T^*}^2+\frac{(\lambda T^*)^2}{64}
\\&
\leq (\ln^- v)^2-\frac{3\lambda T^*}4\ln^- v+m^2K_3{T^*}^2+\frac{(\lambda T^*)^2}{64}.
\end{align*}
As a consequence, we get
\begin{equation}\label{e4.18}
  \E\left([\ln^- I_{u,v}(mT^*)]^2\1_{\Omega_8^*}\right)
\leq (\ln^- v)^2\PP(\Omega_8^*)-\dfrac 34\lambda T^*\ln^- v\PP(\Omega_8^*)+m^2K_3{T^*}^2+\frac{(\lambda T^*)^2}{64}.
\end{equation}
In addition, it follows from Lemma \ref{lem2.4} that
\begin{equation}\label{e4.19}
  \E\left([\ln^- I_{u,v}(mT^*)]^2\1_{\Omega_9^*}\right)
\leq (\ln^- v)^2\PP(\Omega_9^*)+K_1\sqrt{\PP(\Omega_9^*)}mT^*\ln^- v+m^2K_2{T^*}^2.
\end{equation}
Adding side by side \eqref{e4.11}, \eqref{e4.12}, \eqref{e4.18}, \eqref{e4.19}, and using \eqref{e4.4}, we have
\begin{align*}
\E\left([\ln^- I_{u,v}(mT^*)]^2\right)
&\leq (\ln^- v)^2-\left ( \dfrac 34\lambda(1-5\eps^*)-mK_1\sqrt{5\eps^*}\right)T^*\ln^- v+K_4{T^*}^2
\notag\\&
\leq
(\ln^- v)^2-\frac{\lambda T^*}{2}\ln^- v+K_4{T^*}^2,\;\text{for some}\;K_4>0.
  \end{align*}
To end this proof, we consider $v\geq \delta_2^{*}$.
It follows from Lemma \ref{lem2.4} again that
$$
\E\left([\ln^- I_{u,v}(mT^*)]^2\right)
\leq\Big((\ln^- v)^2+K_1\ln^- v+K_2^2\Big)m^2{T^*}^2.
 $$
It is noted that when $v\geq \delta_2^{*}$ then $-\ln v\leq -\ln \delta_2^{*}$. It deduces that $\Big((\ln^- v)^2+K_1\ln^- v\Big)$  and $ \Big((\ln^- v)^2-\frac{\lambda T^*}2\ln^- v\Big)$ are bounded. So, there exists a constant $K_5>K_4$ such that
\begin{equation*}
\E\left([\ln^- I_{u,v}(mT^*)]^2\right)\leq (\ln^- v)^2-\frac{\lambda T^*}2\ln^- v+K_5{T^*}^2\;\; \forall u> H^*, \; v\in(0, \infty).
 \end{equation*}
The proof is complete.
 \end{proof}

 \begin{lem}\label{petite}
  Any compact subset $\K$ is petite for the Markov chain $(S_{u,v}(nmT^*), I_{u,v}(nmT^*))$ $(n\in\N)$.
The irreducibility and aperiodicity of $(S_{u,v}(nmT^*), I_{u,v}(nmT^*))$ $(n\in\N)$ is a byproduct (see e.g., \cite{MT,EN} for definitions).
\end{lem}

\begin{proof}
The proof of this Lemma can be found in \cite[Lemma 2.6]{DANG}.
\end{proof}
\begin{proof}[Proof of Theorem \ref{R>1}]
As in the proof of \cite[Lemma 2.3]{DANG}, the Lemma \ref{lem2.1} implies that
$$\E e^{C_3 t}\big[V_1(S_{u,v}(t),I_{u,v}(t))\big]\leq V_1(u,v)+\dfrac{C_4(e^{C_3t}-1)}{C_3}\;\;\forall t\geq 0.$$
Therefore, there are $P_1>0, P_2>0$  satisfying
\begin{equation}\label{e4.21}
\E V_1\Big(S_{u,v}(mT^*),I_{u,v}(mT^*)\Big)\leq (1-P_1)V_1(u,v)+P_2\,\,\forall (u,v)\in\R^{2,\circ}_+.
\end{equation}
Let $V_*(u,v)=V_1(u,v)+(\ln^- v)^2$. As we introduced in the beginning of this section, Propositions \ref{prop2.3}, \ref{prop2.4} and \eqref{e4.21} allow us to obtain the existences of a compact set $\K\subset \R^{2,\circ}_+$ and two constants $P_1^*>0, P_2^*>0$, which satisfy
\begin{equation}\label{e4.22}
\E V_*(S_{u,v}(mT^*),I_{u,v}(mT^*))\leq V_*(u,v)-P_1^*\sqrt{V_*(u,v)}+P_2^*\1_{\{(u,v)\in \K\}}\,\forall (u,v)\in \R^{2,\circ}_+.
\end{equation}
Combining \eqref{e4.22}, Lemma \ref{petite} and \cite[Theorem 3.6]{JR} yields
\begin{equation}\label{e4.23}
n\|\PP(nmT^*,(u,v),\cdot)-\pi^*\|\to 0 \text{ as } n\to\infty,
\end{equation}
for some invariant probability measure $\pi^*$ of the Markov chain $(S_{u,v}(nmT^*), I_{u,v}(nmT^*))$.
Let $\tau_{\K}=\inf\{n\in\N: (S_{u,v}(nmT^*), I_{u,v}(nmT^*))\in \K\}$. It follows from the proof of \cite[Theorem 3.6]{JR} that \eqref{e4.22} implies $\E \tau_{\K}<\infty$. Therefore, the Markov process $(S_{u,v}(t), I_{u,v}(t))$ has a (unique) invariant probability measure $\pi_*$\; see \cite[Theorem 4.1]{WK}.
Which means that $\pi_*$ is also an invariant probability measure of the Markov chain $(S_{u,v}(nmT^*), I_{u,v}(nmT^*))$.
In light of \eqref{e4.23}, we must have $\pi_*=\pi^*$, or equivalently, $\pi^*$ is an invariant measure of
the Markov process $(S_{u,v}(t), I_{u,v}(t))$.

In the proofs, we use the function $(\ln^- v)^2$ for the sake of simplicity.
In fact, we can treat $(\ln^- v)^{1+\hat q}$ for any small $\hat q\in(0,1)$ in the same manner.
In more details, by the same arguments, we can obtain that there are $m_{\hat q}$, $T^*_{\hat q}$, $P_{1,\hat q}^*, P_{2,\hat q}^*>0$, a compact set $\K_{\hat q}$ satisfying
\begin{equation*}
\E V_{\hat q}(S_{u,v}(m_{\hat q}T^*_{\hat q}),I_{u,v}(m_{\hat q}T^*_{\hat q}))\leq V_{\hat q}(u,v)-P_{1,\hat q}^*[V_{\hat q}(u,v)]^{\frac1{1+{\hat q}}}+P_{2,\hat q}^*\1_{\{(u,v)\in \K_{\hat q}\}}\,\forall (u,v)\in\R_+^{2,\circ},
\end{equation*}
where $V_{\hat q}(u,v)=V_1(u,v)+(\ln^- v)^{1+{\hat q}}.$
By applying \cite[Theorem 3.6]{JR}, we obtain
$$n^{1/{\hat q}}\|\PP(nm_{\hat q}T^*_{\hat q},(u,v),\cdot)-\pi^*\|\to 0 \text{ as } n\to\infty.$$
Since $\|\PP(t,(u,v),\cdot)-\pi^*\|$ is decreasing in $t$, we easily deduce
$$t^{q^*}\|\PP(t,(u,v),\cdot)-\pi^*\|\to 0 \text{ as } t\to\infty$$
with $q^*=1/{\hat q}\in(1,\infty)$.
It follows from \cite[Theorem 8.1]{MT2} or \cite{WK}, we get the strong law of large number.
\end{proof}

\section{Numerical Examples and Discussion}\label{discuss}
\subsection{Numerical Examples}
In this section, let us give some numerical examples to illustrate our obtained results. We consider the following equation
\begin{equation}\label{example}
\begin{cases}
dS(t)=\Big[a_1-b_1 S(t)
-\dfrac{cS(t)I(t)}{1+S(t)+I(t)}\Big]dt + \sigma_1S(t)dB_1(s)-\dfrac{mS(t)I(t)}{1+S(t)+I(t)}dB_3(s),\\[1ex]
dI(t)=\Big[-b_2+cS(t)I(t)+\dfrac{cS(t)I(t)}{1+S(t)+I(t)}\Big]dt+\sigma_2I(t) mB_2(t)+\dfrac{mS(t)I(t)}{1+S(t)+I(t)}dB_3(s),
\end{cases}
\end{equation}
and the corresponding equation on boundary
\begin{equation*}
d\varphi (t)=\big[a_1-b_1\varphi(t)\big]dt+\sigma_1\varphi(t)dB_1(t).
\end{equation*}
\begin{example}\label{ex-1}

Consider \eqref{example} with parameters 
$a_1=3;
b_1=1; 
b_2=1;
c=1;
\sigma_1=1; 
\sigma_2=1; 
m=1$. 
The direct calculation yields  $\lambda \approx -3.01$ (equivalent $R<1$).  As is seen  from  Figure \ref{comparison} that  although $S(t),\varphi(t)$ start from a same initial value, the graph of $S(t)$ neither lies above the graph of $\varphi(t)$ nor lies below. This means that the comparison theorem is no longer valid in the model with multi noises. However, in view of Theorem \ref{R<1}, $I(t)$ tends to 0 regardless initial values while $S(t)$ converges to an ergodic process $\varphi(t)$. 
\begin{figure}[hpt]
\centering
\includegraphics[width=7cm, height=5.5cm]{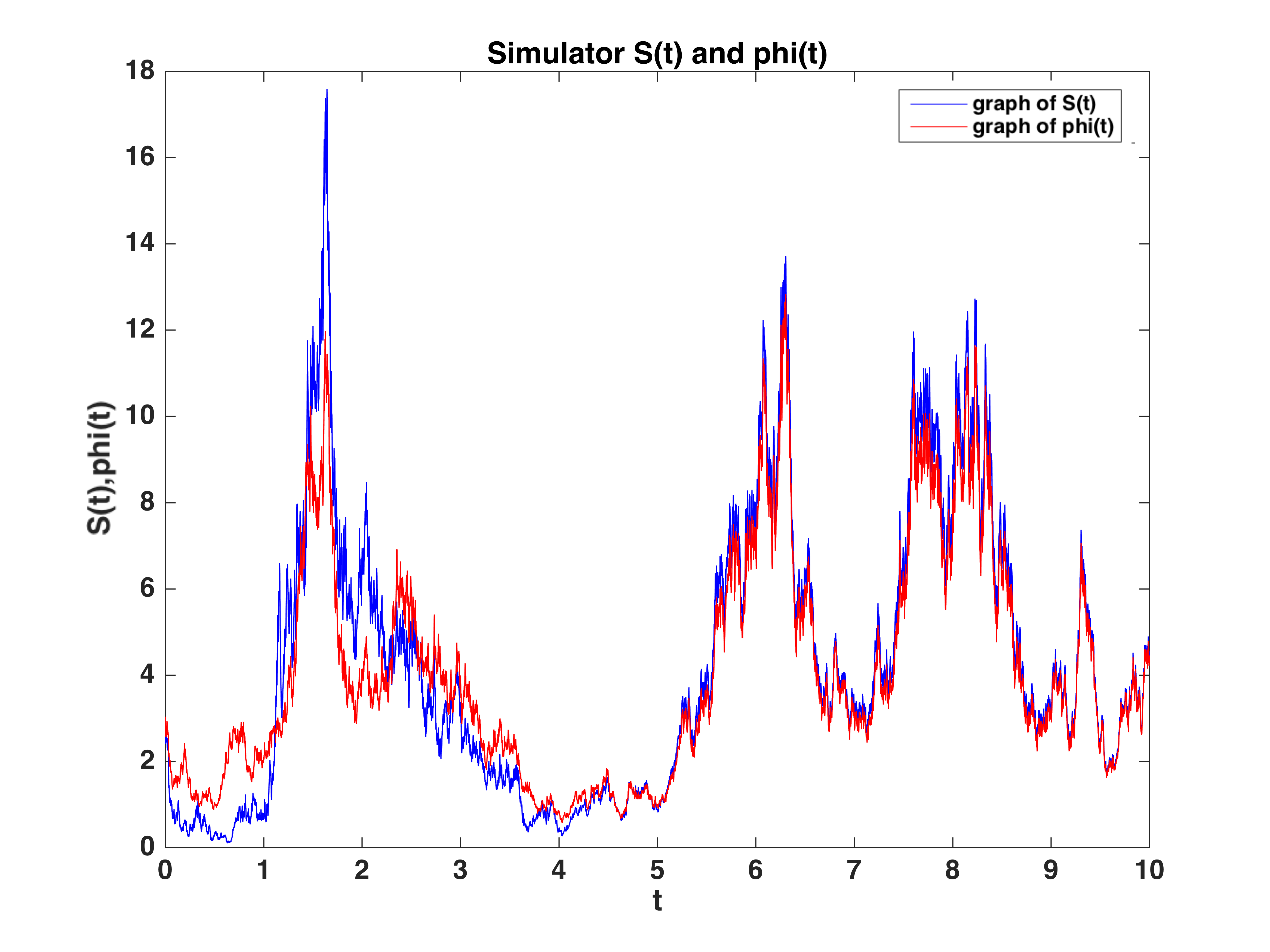}
\includegraphics[width=7.5682cm, height=5.5cm]{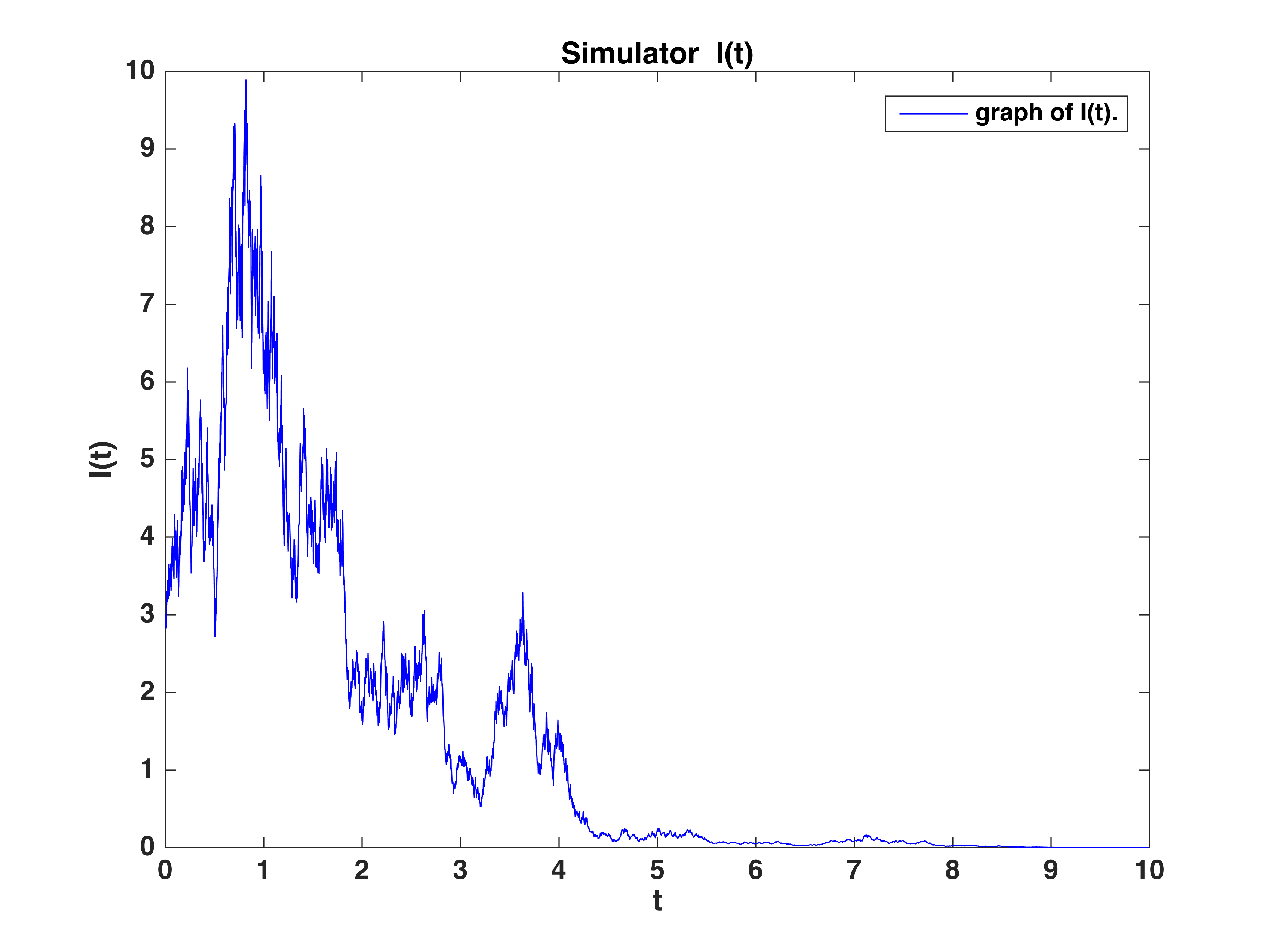}
\caption{Trajectories of $S(t)$ (blue line) and $\varphi(t)$ (red line) and $I(t)$  in Example \ref{ex-1} .}\label{comparison}
\end{figure}
\end{example}

We show an other example where $\lambda>0$ (equivalent $R>1)$. 
\begin{example}\label{ex-2}
Let 
$a_1=10;
b_1=1; 
b_2=1;
c=6;
\sigma_1=1; 
\sigma_2=1; 
m=1.
$
Calculating directly obtains $\lambda\approx 3.3611$, which means that the system \eqref{example} is permanent. Trajectories of $S(t), I(t)$ are shown in Figure \ref{f-R>1}. Moreover, the system \eqref{example} has a unique invariant measure $\pi^*$ concentrated on $\R_+^{2,\circ}$. We draw the empirical density of $\pi^*$ and phase portrait $(S(t),I(t))$ in Figure  \ref{f-R>1.1}.
\begin{figure}[hpt]
\centering
\includegraphics[width=7.5682cm, height=5.5cm]{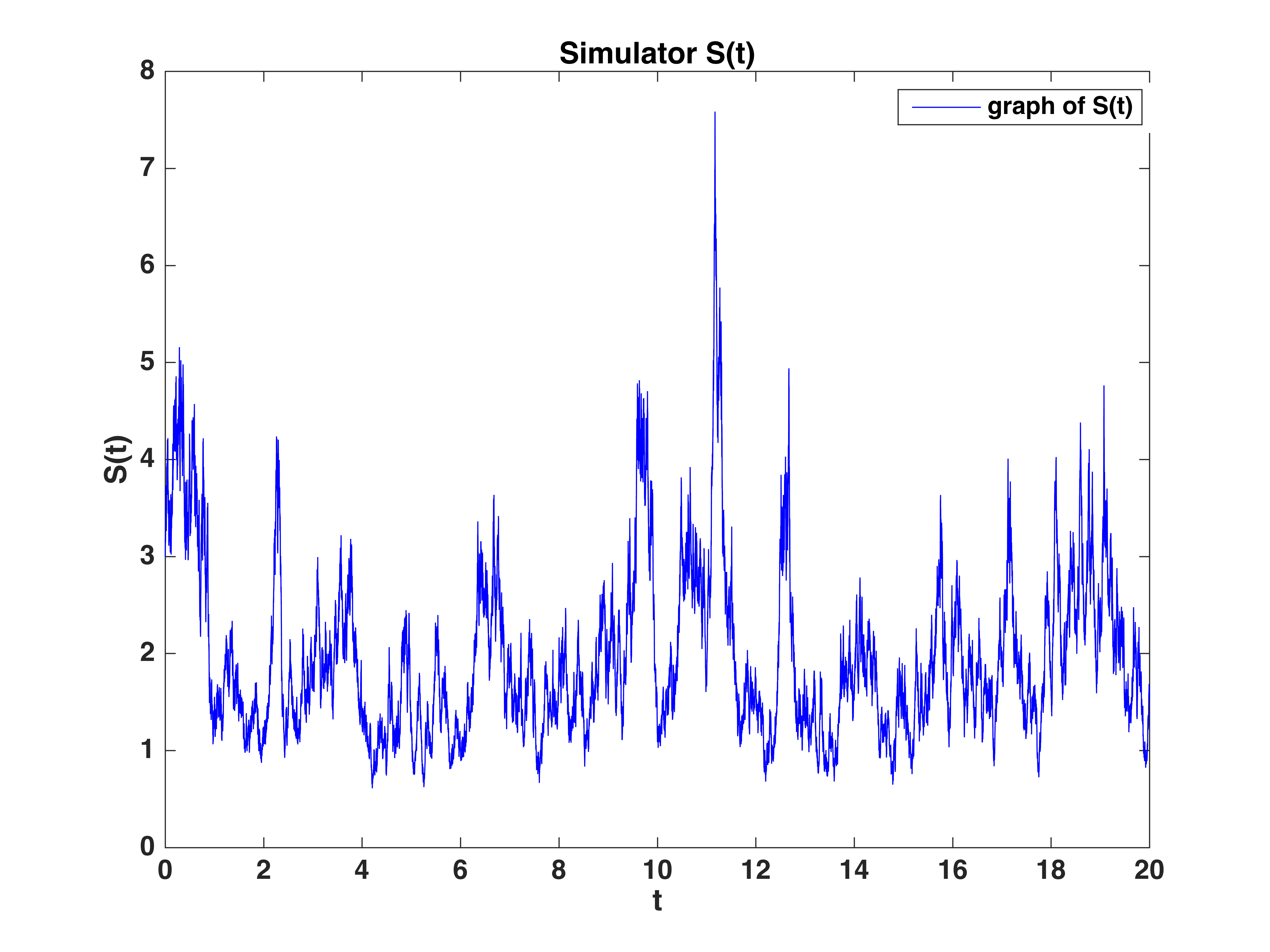}
\includegraphics[width=7.5682cm, height=5.5cm]{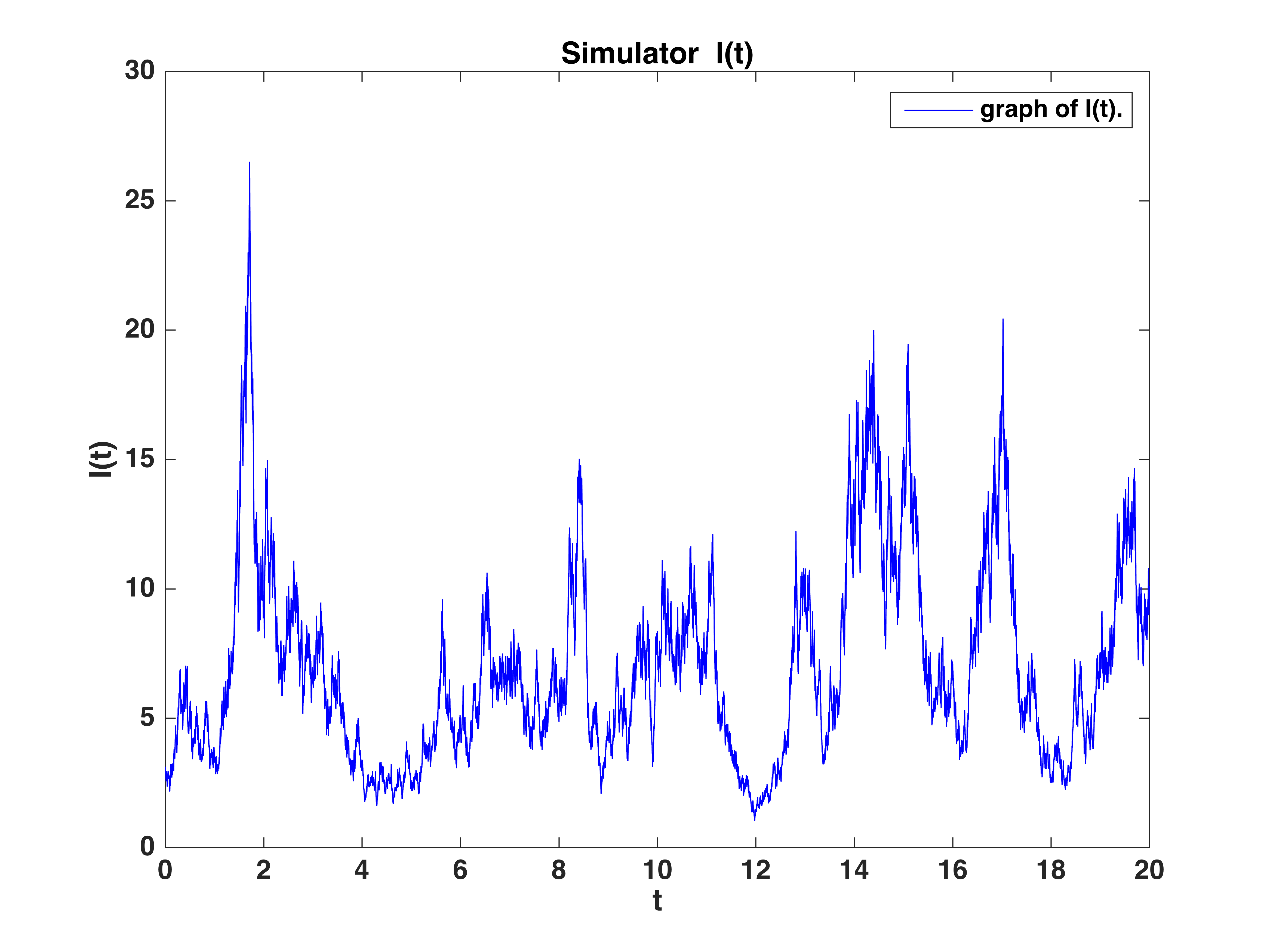}
\caption{Trajectories of $S(t), I(t)$ in Example \ref{ex-2}.}\label{f-R>1}
\end{figure}
\begin{figure}[h]
\centering
\includegraphics[width=7cm, height=5.5cm]{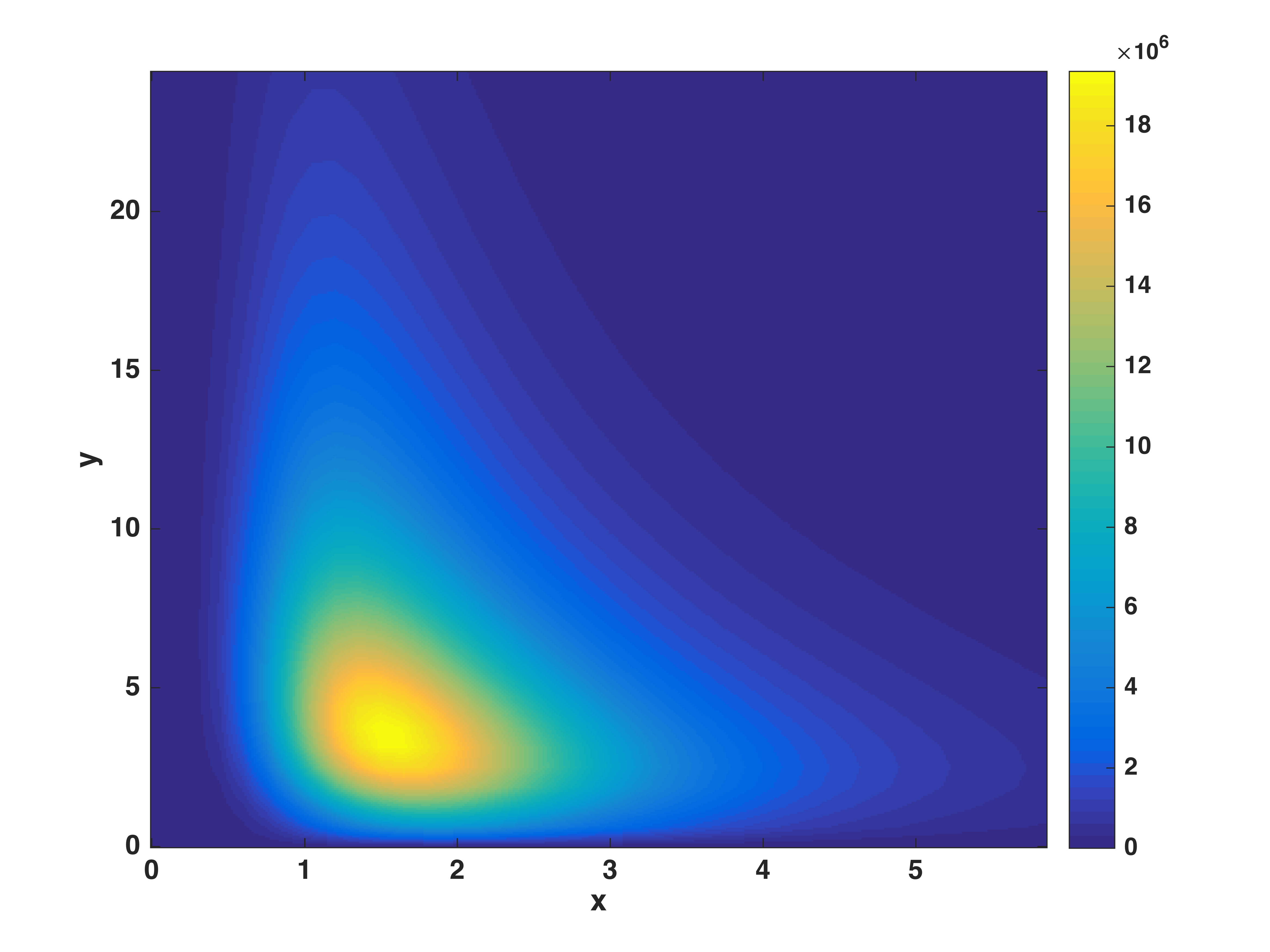}
\includegraphics[width=7cm, height=5.5cm]{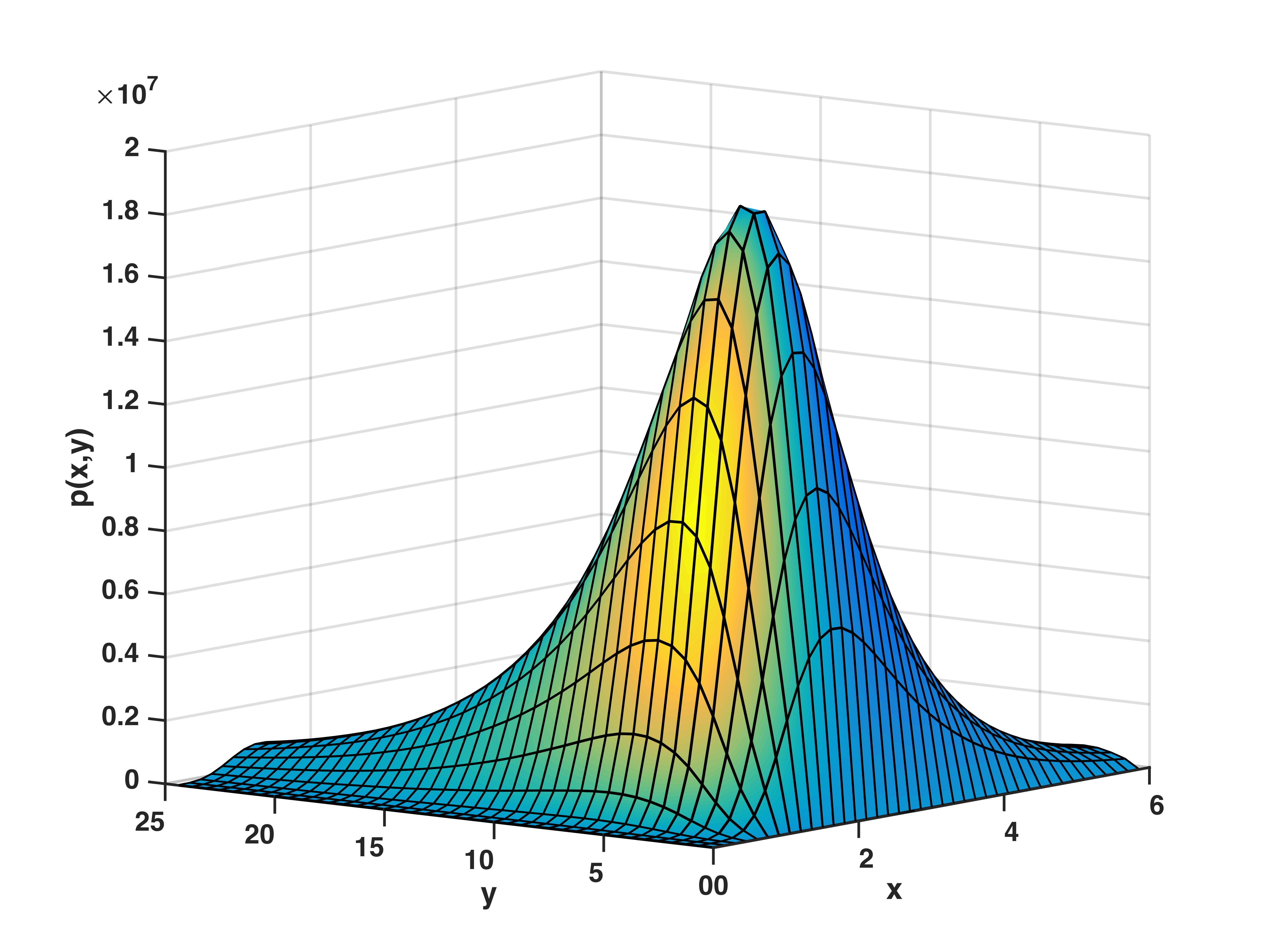}
\caption{Empirical phase and density of $\pi^*$ in Example \ref{ex-2} in 2D and 3D settings respectively.}\label{f-R>1.1}
\end{figure}\end{example}
\subsection{Discussion}
We discuss Theorem \ref{R<1} by providing alternative results in some special cases.
In case of $g(s,i)=0$, we can use the comparison Theorem \cite[Theorem 1.1, p.437]{IW} to obtain that $S_{u,v}(t)\leq \varphi_u(t)\a.s$ Therefore, we can always dominate $S_{u,v}(t)$ by the ergodic process $\varphi_u(t)$ and some estimates can be simplified. 
\begin{thm}
In case of $g(s,i)=0$, the results in Theorem \ref{R<1} hold without the condition $f(s,0)-\dfrac12 g^2(s,0)$ is monotonic.
\end{thm} 
\begin{proof}
The proof can be similarly obtained as in \cite[proof of Theorem 2.1]{NHU2}.
\end{proof}
If the function $\dfrac {F(s,i)}{s}=\dfrac{if(s,i)}{s}$ is uniformly bounded, for example, the Beddington–DeAngelis incidence rate or the nonlinear functional response with $l=h$, we also obtain the results as in Theorem \ref{R<1}.
\begin{thm}
The results in Theorem \ref{R<1} hold if we replace the condition $f(s,0)-\dfrac 12g^2(s,0)$ is monotonic by the condition $\dfrac{F(s,i)}{s}$ is uniformly bounded.
\end{thm}
\begin{proof}
The proof is similar to \cite{NHU1}. 
\end{proof}

\end{document}